\documentclass[letterpaper,11pt,reqno]{amsart}

\usepackage{diagmac2} 

\usepackage{amsmath, amssymb, amscd, MnSymbol,mathrsfs}
\DeclareMathAlphabet{\mathbbm}{U}{bbm}{m}{n}

\newtheorem{theorem}{Theorem}[section] 

\newtheorem{prop}[theorem]{Proposition}
\newtheorem{cor}[theorem]{Corollary}
\newtheorem{lemma}[theorem]{Lemma}

\theoremstyle{plain}

\newtheorem{example}{Example}

\theoremstyle{remark}
\newtheorem{remark}{Remark}[section] 

\theoremstyle{definition}
\newtheorem{definition}{Definition}[section] 

\def\mathbi#1{\textbf{\em #1}}

\DeclareMathAlphabet{\mathpzc}{OT1}{pzc}{m}{it}

\newcommand{\Vars}{\mathcal{P}}
\newcommand{\impl}{\rightarrow}
\newcommand{\bydef} {:=}                      

\newcommand{\dimpl}{\Rightarrow}
\newcommand{\Frm}{\mathsf{Fm}}
\newcommand{\AlgA}{\mathscr{A}}              
\newcommand{\AlgB}{\mathscr{B}}						
\newcommand{\AlgC}{\mathscr{C}}      
\newcommand{\AlgF}{\mathscr{F}}
\newcommand{\AlgT}{\mathscr{T}}
\newcommand{\algA}{\mathsf{A}}
\newcommand{\algB}{\mathsf{B}}
\newcommand{\algC}{\mathsf{C}}  
\newcommand{\algG}{\mathsf{G}}    
\newcommand{\algT}{\mathsf{T}}    
\newcommand{\alga}{\mathsf{a}}
\newcommand{\algb}{\mathsf{b}}
\newcommand{\algc}{\mathsf{c}}
\newcommand{\algd}{\mathsf{d}}
\newcommand{\algg}{\mathsf{g}}
\newcommand{\ua}{\underline{\alga}}
\newcommand{\ub}{\underline{\algb}}
\newcommand{\uc}{\underline{\algc}}

\newcommand{\eqv}{\leftrightarrow}
\newcommand{\one}{\mathbf{1}}
\newcommand{\zero}{\mathbf{0}}
\newcommand{\opr}{\omega}
\newcommand{\Heyt}{\mathcal{H}}
\newcommand{\FSIHeyt}{\mathcal{H}^\circ}
\newcommand{\CHom}{\boldsymbol{\mathsf{H}}}
\newcommand{\CSub}{\boldsymbol{\mathsf{S}}}

\newcommand{\CFin}{\boldsymbol{\mathsf{F}}}
\newcommand{\CProp}{\boldsymbol{\mathsf{Pr}}}
\newcommand{\Sbs}{\boldsymbol{\mathsf{Sbs}}}
\newcommand{\CRel}{\boldsymbol{\mathsf{R}}}

\newcommand{\Z}{\mathsf{Z}} 
\newcommand{\LogL}{\mathsf{L}}
\newcommand{\classL}{\mathcal{L}}
\newcommand{\classK}{\mathcal{K}}
\newcommand{\classV}{\mathcal{V}}
\newcommand{\classF}{\mathcal{F}}
\newcommand{\classA}{\mathcal{A}}
\newcommand{\classI}{\mathcal{I}}

\newcommand{\classC}{\mathcal{C}}
\newcommand{\eqc}{\mathscr{V}}

\newcommand{\SubHeyt}{\Lambda\mathcal{H}}

\newcommand{\Lang}{\mathbi{L}}
\newcommand{\Con}{\mathcal{C}}
\newcommand{\DS}{\mathsf{S}}
\newcommand{\Sign}{\mathcal{S}}

\newcommand{\lbr}{\langle}
\newcommand{\rbr}{\rangle}
\newcommand{\qvar}{\mathcal{Q}}

\newcommand{\idn}{\mathsf{i}}

\newcommand{\ux}{\underline{x}}

\newcommand{\ut}{\underline{t}}

\newcommand{\ophi}{\overline{\phi}}
\newcommand{\opsi}{\overline{\psi}}
\newcommand{\oV}{\overline{\classV}}

\newcommand{\Ax}{\mathpzc{Ax}}
\newcommand{\idt}{\mathsf{t}}
\newcommand{\idr}{\mathsf{r}}
\newcommand{\ids}{\mathsf{s}}
\newcommand{\Distr}{\mathcal{D}}

\newcommand{\set}[2]{\{#1 \mid #2\}}
\def\impln#1{\xrightarrow{#1}}

\begin{document}


\title{Characteristic Formulas 50 Years Later (An Algebraic Account)}

\author{Alex Citkin\\
\texttt{Metropolitan Telecommunications, New York}}

\address{Alex Citkin 30 Upper Warren Way Warren, NJ 07059}
\email{acitkin@gmal.com}

\begin{abstract}
The Jankov (characteristic) formulas were introduced by V.~Jankov fifty tears ago in 1963. Nowadays the Jankov (or frame) formulas are used in virtually every branch of propositional logic: intermediate, modal, fuzzy, relevant, many-valued, etc. All these different logics have one thing in common: in one form or the other, they admit the deduction theorem. From a standpoint of algebraic logic it means that their corresponding varieties have a ternary deductive (TD) term. It is natural to extend the notion of characteristic formula to such varieties and, thus, apply this notion to an even broader class of logics, namely, to the logics which algebraic semantic is a variety with a TD term.  
\end{abstract}


\maketitle

\section{Introduction}

Fifty years ago in his pioneering work \cite[for more details cf. \cite{Jankov_1969}]{Jankov_1963_ded}, V.~Jankov\footnote{Sometimes the transcription ``Yankov'' (e.g. \cite{Skura1989,Hirsch_Hodkinson_Kurucz_2002,Hodkinson_Wolter_Zakharyaschev}) is being used. Even though the transcription ``Yankov'' perhaps is more precise we will be using more commonly accepted transcription ``Jankov''.} introduced a notion of characteristic formula of finite implicative structure. The Jankov formulas provide a relatively simple way of constructing the independent sets of formulas (cf. Section \ref{indsect} for the definition) and, thus, the infinite sets of logics with certain properties. Using this approach in \cite{Jankov_1968}, V.~Jankov constructed the infinite independent sets of intuitionistic (propositional) formulas and proved that there is a continuum of intermediate logics and, hence, there are intermediate logics which are not finitely axiomatizable. Also, he proved that not every intermediate logic enjoys the finite model property.

For some time the technique developed by Jankov in his short 1963 paper \cite{Jankov_1963_ded} went unnoticed. And even a more comprehensive paper \cite{Jankov_1969} published by Jankov in 1968, in which the technique was described in detail, was not immediately appreciated. But this was about to change.  

Independently, in 1968 D.~de~Jongh in his Ph.D. thesis \cite{deJongh_Th} introduced - for intuitionistic frames - a notion of a frame formula that posses the same properties as the Jankov formula. This inspired a separate line of research of different flavors of the frame and subframe formulas (see, for instance, \cite{deJongh_Yanng,Bezhanishvili_N_PhD}). 

Independently, K.~Fine in \cite{Fine_Asc_1974} introduced a notion of a frame formula for \textbf{S4}-frames. As V.~Jankov did, he constructed an infinite independent set of \textbf{S4}-formulas and proved that there is a continuum of normal extensions of \textbf{S4} and, hence, there exist normal extensions of \textbf{S4} which are not finitely axiomatizable. Also K.~Fine proved that there exist normal extensions of \textit{S4} which lacks the finite model property. The definition of the Jankov formula was extended to modal algebras by W.~Routenberg \cite{Rautenberg_Book_1979,Rautenberg_1980} and W.~Blok \cite{Blok_PhD}. 

Different authors use different names for these formulas: Jankov formula, Jankov-de Jongh formula, Jankov-Fine formula. etc. To avoid the confusion we will be using the term ``Jankov formula'' only for the formulas defined by a diagram of a finite subdirectly irreducible algebra, and  for all the different flavors of this notion we will use the original term suggested by Jankov in \cite{Jankov_1963_ded}: characteristic formula.

Nowadays the Jankov (or frame) formulas are used in virtually every branch of propositional logic: intermediate, modal, fuzzy, relevant, many-valued, etc. All these different logics have one thing in common: in one form or the other, they admit the deduction theorem. From a standpoint of algebraic logic it means that their corresponding varieties have a ternary deductive (TD) term \cite{Blk_Pgz_3}. It is natural to extend the notion of characteristic formula to such varieties and, thus, apply this notion to an even broader class of logics, namely, to the logics which algebraic semantic is a variety with a TD term. We give such a generalization below. 

By the same token, the use of a TD term gives us a way to overview the developments in this area from the algebraic standpoint (i.e. focusing only on ``algebraic side'' of the Jankov formula). 

\subsection{Background}

We start with presenting the definitions and results from the original Jankov's paper \cite{Jankov_1963_ded} which is not easily accessible and the translation of which contains numerous typos.  

In this section we consider the (propositional) formulas built up of the countable set of propositional variables $\Vars$ and connectives $\land,\lor,\impl,\neg$. If $A$ is a formula, by $\Vars(A)$ we denote a set of all the propositional variables occurring in $A$. The set of all formulas is denoted by $\Frm$. A mapping $\sigma: \Vars \to \Frm$ is called a \textit{substitution}. If $A$ is a formula, $\sigma(A)$ denotes a result of substitution $\sigma$ applied to the formula $A$. If $\AlgA$ is an algebra in the signature $\land,\lor,\impl,\neg$, a mapping $\nu: \Vars \to \AlgA$ is called a \textit{valuation} (or an \textit{assignment}) \textit{in} $\AlgA$. It is easy to see that using a valuation $\nu$ one can compute in the algebra $\AlgA$ the value of any formula $A$, and we will denote this value by $\nu(A)$. 

If $A_1,\dots,A_n$ and $B$ are formulas, then by $A_1,\dots,A_n \Vdash B$ we denote the derivability of $A$ from $A_1,\dots,A_n$ in the intuitionistic propositional calculus (IPC) with substitution. The set of all formulas derivable in IPC, i.e. $\lbrace A: \ \Vdash A \text{ and } A \in \Frm \rbrace$, forms the intuitionistic propositional logic (IPL). In this paper we alway understand logic as a set of formulas closed under a given set of inferences rules, for instance IPL, while a calculus is a way to define the logic as a set of derivable formulas, for instance IPC.  

If a formulas $A$ is derivable from the formulas $A_1,\dots,A_n$ without substitution (but using axiom schemata), we write $A_1,\dots,A_n \vdash A$. By virtue of the deduction theorem, $A_1,\dots,A_n \vdash A$ if and only if $\Vdash A_1 \land \dots \land A_n \impl A$ and $\vdash A$ if and only if $\Vdash A$.

It is customary to use an abbreviation $A \eqv B$ for $(A \impl B) \land (B \impl A)$. We say that formulas $A$ and $B$ are\textit{ equivalent (in IPC)} and we write $A \cong B$ if $\Vdash A \impl B$ and $\Vdash B \impl A$. Clearly $A \cong B$ if and only if $\Vdash A \eqv B$ and if and only if $A \vdash B$ and $B \vdash A$. If $A,B$ are such formulas that $A \Vdash B$ and $B \Vdash A$ we say that $A$ and $B$ are \textit{interderivable} and denote this by $A \sim B$. Obviously,
\[
A \cong B \dimpl A \sim B
\]
(here and later we are using $\dimpl$ as a replacement for ``yields'').

A property $\pi$ of formulas, that is a unary predicate on $\Frm$, we call \textit{d-stable}\footnote{V.~Jankov called it \textit{intuitionistic}.} if $\pi(A)$ yields $\pi(B)$ for all $B$ such that $A \Vdash B$. For instance, $\pi(A)$ can mean that a formulas $A$ is valid in some model of IPC. 

The Heyting algebras are algebraic models (semantic) for IPL (more details can be found in the Section \ref{heytsem}). We recall that \textit{Heyting algebra} is a bounded distributive lattice relative to $\land,\lor$ with pseudo-complement $\neg$ and relative pseudo-complement $\impl$. We denote algebras by script-like capitals $\AlgA, \AlgB,\AlgC, \dots$ and their respective universes by $\algA,\algB, \algC,\dots$. By $|\AlgA|$ we denote the power of algebra $\AlgA$. The top element of a Heyting algebra $\AlgA$ we denote by $\one_\AlgA$ and the bottom element by $\zero_\AlgA$ and we omit the indexes when no confusion arises. If an algebra $\AlgA$ contains the greatest element among elements distinct from $\one$ such an element is called an \textit{opremum} and we denote it by $\opr(\AlgA)$. 

If $\AlgA$ is an algebra (not necessarily Heyting), by $Con(\AlgA)$ we denote the collection of all congruences on $\AlgA$. If $\theta \in Con(\AlgA)$ and $\alga \in \algA$ then $[\alga]_\theta$ is a congruence class relative to $\theta$ containing element $\alga$. For any two elements $\alga,\algb \in \algA$ there is a smallest congruence $\theta(\alga,\algb)$ such that $\alga \equiv \algb \pmod{\theta}$. Congruence $\theta(\alga,\algb)$ is called \cite{GraetzerB} \textit{principal congruence generated by elements} $\alga,\algb$. Likewise, for any finite lists of elements $\ua \bydef \alga_1,\dots,\alga_n$ and $\ub \bydef \algb_1,\dots,\algb_n$ of the same length there is a smallest congruence $\theta(\ua,\ub)$ such that $\alga_i \equiv \algb_i \pmod{\theta}; i=1,\dots,n$. The congruence $\theta(\ua,\ub)$ is called \cite{GraetzerB} \textit{compact (or finitely generated) congruence generated by} $\ua,\ub$. Recall from \cite{GraetzerB} that an algebra is called \textit{subdirectly irreducible (s.i.)} if there is a smallest non-trivial (i.e. distinct from identity) congruence which is called a \textit{monolith}. If $\AlgA$ is an s.i. algebra then $\mu(\AlgA)$ denotes its monolith. 

It is well known that a Heyting algebra $\AlgA$ is s.i. if and only if it contains an opremum and $\mu(\AlgA) = \theta(\omega(\AlgA),\one_\AlgA)$. The variety of all Heyting algebras we denote by $\Heyt$ and the subset of all finite s.i. algebras from $\Heyt$ we denote by $\FSIHeyt$.

Let us note the following property of homomorphisms of Heyting algebras that follows immediately from the definition of monolith.

\begin{prop} \label{prembed} Let $\AlgA$ be an s.i. Heyting algebra and $\AlgB$ be a Heyting algebra. Then a homomorphism $\phi: \AlgA \impl \AlgB$ is an embedding if and only if $\phi(\omega(\AlgA)) \neq \one_\AlgB$.
\end{prop}

A formula $A$ is said to be \textit{valid} in a Heyting algebra $\AlgA$ (in symbols $\AlgA \models A$) if $\nu(A) = \one$ for every valuation $\nu: \Vars \impl \AlgA$. If $\nu$ is a valuation in an algebra $\AlgA$ and $\nu(A) \neq \one$ we say that $\nu$ is a \textit{refutation of} $A$ in $\AlgA$. Clearly, a formula $A$ is valid in an algebra $\AlgA$ if and only if the identity $A \approx \one$ hods in $\AlgA$ (in symbols $\AlgA \models A \approx \one$).   

 If $\classK$ is a class of algebras by $\CSub\classK$ and $\CHom\classK$ denote respectively the classes of all subalgebras and all homomorphic images of algebras from $\classK$. If $\classK$ consists of a single algebra $\AlgA$, we write $\CSub\AlgA$ and $\CHom\AlgA$. By $\CFin(\classK)$ we denote the subset of all finite non-degenerate (i.e. having more than one element) algebras from $\classK$.

For the theory of characteristic formulas the following relation between implication and congruences is crucial. 

\begin{prop} \label{congref} Let $\AlgA$ be an algebra and $\nu$ be a refutation of formula $A \impl B$ in $\AlgA$. Then there is a congruence $\theta$ on $\AlgA$ such that 
\[
\nu(A) \equiv \one \pmod{\theta} \text{ and } \nu(B) \not\equiv \one \pmod{\theta}. 
\] 
\end{prop}

In \cite{Jankov_1969} V.~Jankov observed the following important for us notion that was baptized in \cite{Kuznetsov_Gerchiu}: a formula $A$ is called \textit{pre-true in an algebra} $\AlgA$ if $A$ is invalid in $\AlgA$ but $A$ is valid in every proper subalgebra and any proper homomorphic image of $\AlgA$.

\begin{remark} V.~Jankov was considering pre-true formulas only for finite algebras, while in \cite{Kuznetsov_Gerchiu} it was essential that some infinite algebras can have a pre-true formula.  
\end{remark}

\begin{prop} \label{pre-true} Suppose $A$ is a formula such that $\rm IPC \nvdash A$. Then there is a finite Heyting algebra $\AlgA$ in which formula $A$ is pre-true.
\end{prop}

\begin{proof} Recall that IPC enjoys a \textit{finite model property} (fmp), that is any formula $A$ that is not derivable in IPC is invalid in some finite Heyting algebra. Then one can take a finite algebra of the smallest power in which formula $A$ is invalid. Clearly, formula $A$ is pre-true in such an algebra.  
\end{proof}

Note that the Proposition \ref{pre-true1} holds for any logic with exact algebraic semantic (cf. Section \ref{algsemsec}) that enjoys fmp. 

\begin{prop} \label{pre-true1} Let $A(p_1,\dots,p_n)$ be a formula pre-true in an algebra $\AlgA$. Then the following hods
\begin{itemize}
\item[(a)] If $\nu$ is a refuting valuation then elements $\nu(p_1),\dots,\nu(p_n)$ generate algebra $\AlgA$;
\item[(b)] Algebra $\AlgA$ is s.i.; 
\item[(c)] If $\nu$ is a refuting valuation then $\nu(A) = \omega(\AlgA)$. 
\end{itemize}
\end{prop}

\begin{proof} Trivial.
\end{proof}

Note that (a) and (b) hold for any logics with equivalent algebraic semantic (for the definition refer to the Section \ref{algsemsec}).

\subsection{Characteristic Formulas: Definition}

With every finite s.i. Heyting algebra $\AlgA$ in the following way one can associate a formula $\delta^+(\AlgA)$ in variables $p_\alga; \alga \in \algA$ that we call a \textit{positive diagram formula} of $\AlgA$:
\begin{equation*}
\delta^+(\AlgA) \bydef \bigwedge_{\circ \in \lbrace \land, \lor,\impl\rbrace} \bigwedge_{\alga,\algb \in \algA} ((p_\alga \circ p_\algb) \eqv p_{\alga \circ \algb}) \land \bigwedge_{\alga \in \algA}(\neg p_\alga \eqv p_{\neg \alga}). \label{Diagram}
\end{equation*}

\begin{definition}
For each finite s.i. Heyting algebra  the formula  
\begin{equation}
J(\AlgA) \bydef \delta^+(\AlgA) \impl \omega(\AlgA).  \label{Jankov} 
\end{equation}
is called (cf. \cite{Jankov_1963_ded}) a \textit{Jankov characteristic formula} (Jankov formula for short)\footnote{In \cite{Litak_PhD} the Jankov formula is understood as a positive diagram.} .
\end{definition}

Next, we generalize the notion of Jankov formula by using the notion of finitely presented algebra. 

\begin{definition} \label{finpres}
Let $\AlgA$ be an algebra generated by a finite set $\algG$ of generators. Algebra is called \textit{finitely presented in generators} $\algG$ if there is a finite set of variables $\Vars_0$, a mapping (valuation) $\nu$ of  $\Vars_0$ onto $\algG$ and the formulas $A_1,\dots,A_n$ such that
\begin{itemize}
\item[(fp1)] $\Vars(A_i) \subseteq \Vars_0; i=1,\dots n$;
\item[(fp2)] $\nu(A_i) = \one_\AlgA$ for all $i=1,\dots n$;
\item[(fp3)] $\vdash A_1 \land \dots A_n \impl B$ for any formula $B$ such that $\Vars(B) \subseteq \Vars_0$ and $\nu(B) = \one_\AlgA$.
\end{itemize} 
\end{definition}

The variables $\Vars_0$ are called \textit{defining variables}, the valuation $\nu$ is called \textsl{defining valuation} and formulas $A_1,\dots,A_n$ are called \textit{defining formulas} (in variables $\Vars_0$ and valuation $\nu$). For Heyting algebras we can safely assume that we always have one defining formula. Clearly, if $A_1$ and $A_2$ are two formulas defining the same algebra in the same set of variables and the same valuation, then $\vdash A_1 \eqv A_2$.    

\begin{table}[h]
\begin{tabular}{p{290pt} r}  
\begin{example} \label{z3} Let $\Z_3 = \lbrace \zero,\omega,\one; \land, \lor, \impl, \neg \rbrace$ be a 3-element Heyting algebra. The element $\omega$ generates $\Z_3$. If we take $\Vars_0 = \lbrace p \rbrace$ and $\nu: p \mapsto \omega$ as defining variables and valuation, then the formula $\neg \neg p$ will define algebra $\Z_3$ (in variables $\Vars_0$ and valuation $\nu$), hence, algebra $\Z_3$ is finitely presented.   
\end{example} &
\ctdiagram{
\ctnohead
\cten 0,0,0,-30:{}
\ctv 0,0:{\bullet}
\ctv 0,-15:{\bullet}
\ctv 0,-30:{\bullet}
\ctv 7,0:{1}
\ctv 7,-15:{\omega}
\ctv 7,-30:{0}
\ctv 0,-40:{Fig.1.}
\ctv  0,-50:{Algebra \ \Z_3}
}
\end{tabular}
\end{table}

\begin{example}
It is not hard to see that for any finite non-trivial algebra $\AlgA$ one can take $\Vars_0 = \lbrace p_\alga: \alga \in \algA  \rbrace$, $\nu: p_\alga \mapsto \alga$ as defining variables and valuation and the formula $\delta^+(\AlgA)$ will define algebra $\AlgA$. In other words, any finite algebra is finitely presented, particularly it is finitely presented in a set of generators consisting of all its elements.
\end{example}
 
Perhaps the most important property of the defining formula (that is often taken as a definition of finitely presented algebra) is that in the Definition \ref{finpres} the requirement (fp3) can be replaced (cf. \cite[p.217 Theorem 1]{MaltsevBook})) with the following:\\
  
 \begin{itemize}
\item[(fp3')] For any algebra $\AlgB$ and any mapping $\phi: \algG \impl \AlgB$ if\\
$A(\phi(\nu(p_1)),\dots,\phi(\nu(p_n))) = \one_\AlgB$ 
then the mapping $\phi$ can be extended to a homomorphism $\overline{\phi}: \AlgA \impl \AlgB$. 
\end{itemize}
In other words, the following is true.

\begin{prop}\label{fphom} An algebra $\AlgA$ is finitely presented if and only if (fp1), (fp2) and (fp3') hold.
\end{prop}

From the above Proposition it follows that if formulas $A$ and $B$ define the same algebra, then $A \sim B$. And, as we have noticed earlier, if formulas $A$ and $B$ define the same algebra in the same set of variables and the same valuation, then $A$ and $B$ are equivalent.   

\begin{definition}
Let an s.i. algebra $\AlgA$ be finitely presented by a formula $A$ in a set of variables $\Vars_0$ and  a valuation $\nu$. The formula 
\begin{equation}
	\chi(\AlgA,\Vars_0,\nu) \bydef A \impl B \text{ where } \Vars(B) \subseteq \Vars_0 \text{ and }\nu(B) = \omega(\AlgA) \tag{Char} \label{char}
\end{equation}
we call a \textit{characteristic formula}. 
\end{definition}

Let us note that $\nu$ is a refuting valuation and, hence,
\begin{equation}
\AlgA \not\models \chi(\AlgA,\Vars',\nu). \label{selfrefut}
\end{equation}

\begin{remark}
The formula $B$ from the definition always exists because elements $\nu(p);p \in \Vars'$ generate algebra $\AlgA$.  Neither the defining formula, nor the formula $B$ expressing $\omega(\AlgA)$ are unique. Later we will demonstrate that a characteristic formula is defined uniquely modulo $\sim$; and for given $\Vars', \nu$, a characteristic formula is defined uniquely modulo $\eqv$. 
\end{remark}

\begin{example} \label{wexcl} The formula $\neg \neg p \impl p$ is a characteristic formula of 3-element Heyting algebra $\Z_3$ depicted at the Fig.1. The reader can find (almost a half of the page long) Jankov formula of this algebra in \cite[10.3.2]{Galatos_et_Book}.
\end{example} 

Recall that by Tietze's Theorem \cite[p.222]{MaltsevBook} if an algebra is finitely presented in some set of generators, then it is finitely presentable in any finite set of generators. 

\begin{example} If $\AlgA$ is a finite s.i. Heyting algebra one can take the set $\algg_1,\dots, \algg_n$ of all distinct from $\zero$ $\lor$-irreducible elements as a set of generators: any element of $\AlgA$ can be expressed as a disjunction of $\lor$-irreducible elements. If we define a characteristic formula of $\AlgA$ in variables $\lbrace p_{\algg_1},\dots,p_{\algg_n} \rbrace$ and valuation $\nu: p_{\algg_i} \mapsto \algg_i; i=1,\dots,n$ we obtain the de Jongh formula \cite{deJongh_Th,Bezhanishvili_N_PhD} of $\AlgA$. Similarly for closure algebras we will obtain the Fine formula \cite{Fine_Asc_1974}.
\end{example}

\subsection{Characteristic Formulas: Properties} \label{Jankovprop}

Almost all properties of characteristic formulas can be derived from the following.

\begin{prop}\label{prophom} (comp. \cite[Theorem about Ordering]{Jankov_1969}) Let $\AlgA$ be a finitely presented s.i. algebra and $\AlgB$ be an algebra. Then
\[
\AlgB \not\models \chi(\AlgA) \text{ if and only if } \AlgA \in \CSub\CHom\AlgB.
\]
\end{prop}
\begin{proof} Let $A \impl B$ be a characteristic formula of $\AlgA$ and $\nu$ be a valuation refuting $A \impl B$ in $\AlgB$. By the Proposition \ref{congref} there is a congruence $\theta$ on $\AlgB$ such that 
\[
\nu(A) \equiv \one_\AlgB \pmod{\theta}  \text{ and } \nu(B) \not\equiv \one_\AlgB \pmod{\theta}. 
\]
Let us consider a quotient algebra $\AlgB/\theta$ that consists of congruence classes $[\algb]_\theta; \algb \in \algB$. We have
\[
[\nu(A)]_\theta = [\one_\AlgB]_\theta \text{ and } [\nu(B)]_\theta \neq [\one_\AlgB]_\theta. 
\]
Therefore
\[
A([\nu(p_1)]_\theta, \dots, [\nu(p_n)]_\theta) = \one_{\AlgB/\theta} \text{ and } B([\nu(p_1)]_\theta, \dots, [\nu(p_n)]_\theta) \neq \one_{\AlgB/\theta},
\]
where $p_1,\dots,p_n$ is a list of all variables occurring in $A,B$. By virtue of the Proposition \ref{fphom}, the mapping
\[
\phi: \nu(p_i) \mapsto [\nu(p_i)]_\theta; i=1,\dots,n
\]
can be extended to a homomorphism 
\[\overline{\phi}: \AlgA \impl \AlgB/\theta.
\] 
And for $\phi(\nu(B)) \neq \one_{\AlgB/\theta}$, by the Proposition \ref{prembed}, we can conclude that $\overline{\phi}$ is an isomorphism. Thus $\AlgA$ embeds in $\AlgB/\theta$, that is $\AlgA \in \CSub\CHom\AlgB$. 
\end{proof}

The most important properties of characteristic formulas are presented by the following theorem proved in \cite{Jankov_1963_ded}. 

\begin{theorem}[Jankov Theorem] \label{JankovThrm} Let $\AlgA$ be a finite s.i. Heyting algebra, $A$ be a formula and $\pi$ be a d-stable property. Then
\begin{itemize}
\item[(Ded)] $\AlgA \not\models A$ if and only if $A \Vdash \chi(\AlgA)$;
\item[(Prp)] $\AlgA \models B$ for every formula $B$ such that $\pi(B)$ is true if and only if $\pi(\chi(\AlgA))$ is not true; 
\item[(Gen)] If the algebra $\AlgA$ is generated by its $k$ elements then there is a formula $B$ containing $k$ distinct variables such that $B \sim \chi(\AlgA)$;
\item[(Rft)] For any finite s.i. Heyting algebra $\AlgB$ the following are equivalent
\begin{itemize}
\item[(a)] $\chi(\AlgA) \Vdash \chi(\AlgB)$;
\item[(b)] $\AlgB \not\models \chi(\AlgA)$;
\item[(c)] any formula $B$ refutable in $\AlgA$ is refutable in $\AlgB$.
\end{itemize}
\end{itemize}
\end{theorem}
\begin{proof}
(Gen) follows immediately from the definition of finitely presented algebra and characteristic formula.

The proof of the rest of the items follows straight from the Proposition \ref{prophom} and the fact that for any formulas $A$ and $B$ 
\[
A \Vdash B \text{ if and only if } \AlgA \models A  \text{ yields } \AlgA \models B \text{ in any Heyting algebra } \AlgA.
\] 

(Ded) If $\AlgA \not\models A$ then, by virtue of the Proposition \ref{prophom}, the formula $A$ is refutable in any algebra in which $\chi(\AlgA)$ is refutable, that is $A \Vdash \chi(\AlgA)$. The converse statement is a trivial consequence of \eqref{selfrefut}.

(Rft)
(a) $\dimpl$ (b). If $\chi(\AlgA) \Vdash \chi(\AlgB)$, by \eqref{selfrefut} $\AlgB \not\models \chi(\AlgB)$ therefore $\AlgB \not\models \chi(\AlgA)$.
 
(b) $\dimpl$ (c). If $\AlgB \not\models \chi(\AlgA)$, then by the Proposition \ref{prophom} $\AlgA \in \CSub\CHom\AlgB$. Hence, any formula refutable in $\AlgA$ is refutable in $\AlgB$. 

(c) $\dimpl$ (a). By \eqref{selfrefut}, $\AlgA \not\models \chi(\AlgA)$, hence by (c) $\AlgB \not\models \chi(\AlgA)$ and, by the Proposition \ref{prophom}, 
\[
\AlgA \in \CSub\CHom\AlgB.
\] 
Assume $\AlgC$ is an algebra and $\AlgC \not\models \chi(\AlgB)$. Then
\[ 
\AlgB \in \CSub\CHom \AlgC,
\]
 and
\[
\AlgA \in \CSub\CHom\AlgB \subseteq \CSub\CHom \AlgC.
\]
Thus $\AlgC \not\models \chi(\AlgA)$, that is $\chi(\AlgA)$ is refutable in any algebra where $\chi(\AlgB)$ is refutable. Hence, $\chi(\AlgA) \Vdash \chi(\AlgB)$.

(Prp) Suppose $\AlgA \models B$ for all formulas $B$ such that $\pi(B)$ is true. Then $\pi(\chi(\AlgA))$ is not true for \eqref{selfrefut}.

Conversely, assume for contradiction that there is a formula $B$ satisfying $\pi$ such that $\AlgA \not\models B$. Then by (Ded) $B \Vdash \chi(\AlgA)$ and, by the definition of d-stable property, $\pi(\chi(\AlgA))$ is true.
\end{proof}

\begin{example} (stated without a proof in  \cite[Theorem 3(a)]{Jankov_1963_sup} and proved in \cite[Theorem 5.3]{Troelstra_Intermediate_1965}, \cite[Theorem 1]{Jankov_Extension_1968}) As we saw in the Example \ref{wexcl}, the formula $\neg \neg p \impl p$ is a characteristic formula of the algebra $\Z_3$. By (Ded) for every formula $A$ if $\Z_3 \not\models A$ then $A \Vdash (\neg \neg p \impl p)$. Thus, for any formula $A$ valid in classical propositional logic ($\rm CPL$), the calculus $\rm IPC + A$ defines $\rm CPL$ if and only if $\Z_3 \not\models A$. 
\end{example}

In the same paper \cite{Jankov_1963_ded} on the set $\FSIHeyt$ V.~Jankov introduced a quasi-order:
\[
\AlgA \leq \AlgB \bydef \AlgB \not\models \chi(\AlgA).
\]
He also observed that the above quasi-order is, in fact, a partial order. The pure algebraic meaning of the introduced relation can be seen from the following very important for applications property of characteristic formulas that was not explicitly stated in \cite{Jankov_1963_ded} but was proved in \cite{Jankov_1969} as a following 
\begin{theorem}[About Ordering] \label{thrmord} Let $\AlgA$ and $\AlgB$ be finite s.i. Heyting algebras. Then the following are equivalent 
\begin{itemize}
\item[(a)] $\AlgA \leq \AlgB$;
\item[(b)] $\AlgB \not\models \chi(\AlgA)$;
\item[(c)] every formula refutable in $\AlgA$ is refutable in $\AlgB$;
\item[(d)] $\AlgA \in \CSub\CHom\AlgB$.  
\end{itemize}
\end{theorem}

The proof of the Theorem \ref{thrmord} follows immediately from the Jankov Theorem and Proposition \ref{prophom}.

\begin{remark} In \cite{Jankov_1963_ded} V.~Jankov is considering not only Heyting algebras but also the $\lbrace \land,\lor,\impl\rbrace$ and $\lbrace \land,\impl\rbrace$-reducts as well as these reducts endowed by a constant $\zero$ without any additional axioms for $\zero$.
\end{remark}

Let us note that neither in the proof of the Proposition \ref{prophom}, nor in the proof of the Jankov Theorem or the Theorem \ref{thrmord}, we were using finiteness of the s.i algebra. Thus, the statements hold for any finitely presented s.i. algebras. But, as it was observed in \cite[Theorem 1.3]{Citkin_PhD} (see also \cite[Proposition 2.1]{Butz_1998}), any finitely presented Heyting algebra is finitely approximated, hence, every finitely presented s.i. Heyting Algebra is finite (it is true for any finitely approximated variety with equationally definable principal congruences (EDPC) \cite[Corollary 3.3]{Blk_Pgz_1}). However, this observation becomes important if we are considering not finitely approximated varieties that may contain infinite finitely presented s.i. algebras.

\subsection{Independent Sets} \label{indsect}

The quasi-order that was introduced in the previous section and its properties formulated in the Theorem \ref{thrmord} give us the means for constructing sets of logics obeying a given d-stable property. In \cite{Jankov_1968} V.~Jankov observed that there is a continuum of intermediate logics. In order to prove this he introduced a notion of strongly independent logics\footnote{In \cite{Troelstra_Intermediate_1965} A.~Trolstra defined independence of logics $\LogL_1$ and $\LogL_2$ as their incomparability, i.e. $\LogL_1 \nsubseteq \LogL_2$ and $\LogL_2 \nsubseteq \LogL_1$.}: a set of logics $\classL = \set{ \LogL_i}{i\in I}$ he called \textit{strongly independent} if neither logic $\LogL_i$ is included in the logic generated by the rest of the logics from $\classL,$ that is $\LogL_i \notin \cup'\set{\LogL_j}{j \neq i. j\in J }$ (where $\cup'$ is a closed union). One can easily construct a strongly independent set of logics if we have an independent set of formulas: a set of formulas $\classF = \set{ A_i}{ i\in I }$ is called \textit{independent} if neither formula $A_i$ is derivable from the rest of the formulas from $\classF$, that is $\classF \setminus \lbrace A_i \rbrace \nVdash A_i$ for every $i \in I$. Indeed, it is not hard to see that the logics defined by distinct subsets of an independent set of formulas $\classF$ form a strongly independent set of logics. The Theorem \ref{thrmord} states that a set of characteristic formulas $\classA = \set{ \chi(\AlgA_i)}{ i \in I}$ is independent if and only if for every $i \neq j; i,j \in I$ we have $\AlgA_i \not\leq \AlgA_j$ and $\AlgA_j \not\leq \AlgA_i$ , i.e. the set $\classA$ forms an \textit{anti-chain} relative to quasi-order. Thus, given an anti-chain of the finite s.i. algebras,  one can construct an independent set of formulas, namely the set of characteristic formulas of these algebras. 

Before we give an example of an anti-chain in $\FSIHeyt$ let us observe that by virtue of the Theorem \ref{thrmord} and the definition of the quasi-ordering
\begin{equation}
\AlgA \leq \AlgB \text{ if and only if } \AlgA \in \CSub\CHom \AlgB. \label{qasiorder}
\end{equation}

\begin{example} \label{exindep} The Heyting algebras depicted at Fig 2. by the diagrams and frames form an infinite anti-chain. Hence, the characteristic formulas of these algebras form an independent set $\classF$. Therefore any two distinct subsets of $\classF$ define the distinct intermediate logics. Thus, there is a continuum of intermediate logics (cf. \cite[Corollary 1]{Jankov_1968}) and there are not finitely axiomatizable intermediate logics (cf. \cite[Corollary 2 ]{Jankov_1968}), in fact, also a continuum.
  
\[
\ctdiagram{
\ctnohead
\ctinnermid
\ctel -100,15,-85,0:{}
\ctel -85,0,-55,30:{}
\ctel -100,15,-70,45:{}
\ctel -70,15,-85,30:{}
\ctel -55,30,-70,45:{}
\ctel -70,45,-70,75:{}
\ctv -85,0:{\bullet}
\ctv -100,15:{\bullet}
\ctv -70,15:{\bullet}
\ctv -55,30:{\bullet}
\ctv -85,30:{\bullet}
\ctv -70,45:{\bullet}
\ctv -70,60:{\bullet}
\ctv -70,75:{\bullet}
\ctv -85,-15:{\Z_7+\Z_2}
\ctel 0,15,15,0:{}
\ctel 15,0,45,30:{}
\ctel 0,15,30,45:{}
\ctel 30,15,0,45:{}
\ctel 45,30,15,60:{}
\ctel 0,45,15,60:{}
\ctel 15,60,15,90:{}
\ctv 15,0:{\bullet}
\ctv 0,15:{\bullet}
\ctv 30,15:{\bullet}
\ctv 45,30:{\bullet}
\ctv 15,30:{\bullet}
\ctv 30,45:{\bullet}
\ctv 0,45:{\bullet}
\ctv 15,60:{\bullet}
\ctv 15,75:{\bullet}
\ctv 15,90:{\bullet}
\ctv 15,-15:{\Z_9+\Z_2}
\ctel 100,15,115,0:{}
\ctel 115,0,145,30:{}
\ctel 100,15,145,60:{}
\ctel 130,75,145,60:{}
\ctel 130,15,100,45:{}
\ctel 145,30,115,60:{}
\ctel 100,45,130,75:{}
\ctel 130,75,130,105:{}
\ctv 115,0:{\bullet}
\ctv 100,15:{\bullet}
\ctv 130,15:{\bullet}
\ctv 145,30:{\bullet}
\ctv 115,30:{\bullet}
\ctv 130,45:{\bullet}
\ctv 100,45:{\bullet}
\ctv 115,60:{\bullet}
\ctv 130,75:{\bullet}
\ctv 130,90:{\bullet}
\ctv 130,105:{\bullet}
\ctv 130,75:{\bullet}
\ctv 145,60:{\bullet}
\ctv 115,-15:{\Z_{11}+\Z_2}
\ctv 160,30:{\cdot}
\ctv 165,30:{\cdot}
\ctv 170,30:{\cdot}
\ctv 175,30:{\cdot}
\ctel -100,-45,-85,-60:{}
\ctel -85,-75,-85,-60:{}
\ctel -85,-60,-70,-45:{}
\ctel -70,-45,-70,-30:{}
\ctv -100,-45:{\bullet}
\ctv -70,-30:{\bullet}
\ctv -70,-45:{\bullet}
\ctv -85,-60:{\bullet}
\ctv -85,-75:{\bullet}
\ctel 0,-45, 15,-60:{}
\ctel 15,-75, 15,-60:{}
\ctel 15,-60, 30,-45:{}
\ctel 30,-45,30,-30:{}
\ctel 30,-45,0,-30:{}
\ctel 0,-45,0,-30:{}
\ctv 0,-30:{\bullet}
\ctv 0,-45:{\bullet}
\ctv 30,-30:{\bullet}
\ctv 30,-45:{\bullet}
\ctv 15,-60:{\bullet}
\ctv 15,-75:{\bullet}
\ctel 100,-60, 115,-75:{}
\ctel 115,-90,115,-75:{}
\ctel 115,-75,130,-60:{}
\ctel 130,-60,130,-30:{}
\ctel 100,-60,100,-45:{}
\ctel 100,-45,130,-60:{}
\ctel 100,-60,130,-30:{}
\ctv 100,-60:{\bullet}
\ctv 130,-45:{\bullet}
\ctv 130,-60:{\bullet}
\ctv 115,-75:{\bullet}
\ctv 115,-90:{\bullet}
\ctv 100,-45:{\bullet}
\ctv 130,-30:{\bullet}
\ctv 160,-60:{\cdot}
\ctv 165,-60:{\cdot}
\ctv 170,-60:{\cdot}
\ctv 175,-60:{\cdot}
}
\]

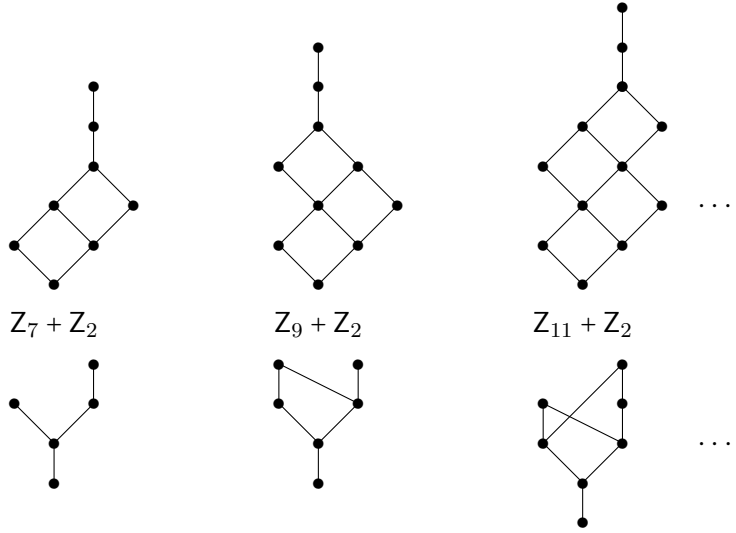
\begin{center} Fig.2 An anti-chain in $\FSIHeyt$ \end{center}
\end{example}

\begin{remark} In \cite{Jankov_1968} V.~Jankov is using a different anti-chain. The anti-chain from the above example was also used in \cite{Blok_Continuum_1977} by W.~Blok in order to construct a continuum of intermediate logics without fmp.   
\end{remark}

\subsection{An Algebraic View at Intermediate Logics} \label{heytsem}

As we already mentioned, the Heyting algebras are the (algebraic) models for intermediate logics. With each intermediate logic $\LogL$ we can associate a set $\eqc(\LogL)$ of all the Heyting algebras in which all the formulas from $\LogL$ are valid, that is
\[
\eqc(\LogL) = \set{ \AlgA}{ (\forall A \in \LogL) \AlgA \Vdash A, \AlgA \in \Heyt }.
\]

It is not hard to see that the set $\eqc(\LogL)$ is closed under direct products, homomorphisms and subalgebras. Thus, $\eqc(\LogL)$ is a variety.

On the other hand, with each variety $\classV$ of Heyting algebras we can associate a logic $\LogL(\classV)$ consisting of all the formulas valid in every algebra from $\classV$:
\[
\LogL(\classV) = \set{ A}{ (\forall \AlgA \in \classV) \AlgA \Vdash A }.
\]

The class of all varieties of Heyting algebras, i.e. the class $\SubHeyt$ of all subvarieties of $\Heyt$, forms a complete lattice relative to set intersection $\cap$ and closed union $\classV_1 \cup' \classV_2 = \eqc(\classV_1 \cup \classV_2)$. The set \textbf{Ext}IPL of all extensions of IPL also forms a complete lattice relative to set intersection $\cap$ and closed union $\LogL_1 \cup' \LogL_2 = \LogL(\LogL_1 \cup \LogL_2)$. Recall that the mappings
\[
\phi: \LogL \impl \eqc(\LogL)  \text{ and } \psi: \classV \impl \LogL(\classV) 
\] 
are the dual isomorphisms between \textbf{Ext}IPL and $\SubHeyt$ and $\psi = \phi^{-1}$.

In general, if $\Frm' \subseteq \Frm$ is a set of formulas, by $\eqc(\Frm')$ we denote the variety of Heyting algebras in which all the formulas from $\Frm'$ are valid and $\LogL(\Frm') \bydef \set{ B}{ \Frm' \Vdash B; B \in \Frm }$. We write $\eqc(A)$ and $\LogL(A)$ in the case when $\Frm'$ consists of a single formula $A$. If $\classK \subseteq \Heyt$ is a class of algebras by $\eqc(\classK)$ we denote a variety generated by algebras $\classK$ and by $\LogL(\classK)$ we denote a logic that consists of all the formulas valid in each algebra from $\classK$. We write $\eqc(\AlgA)$ and $\LogL(\AlgA)$ in the case when $\classK$ consists of a single algebra $\AlgA$.

Next, we observe the very important property of characteristic formulas: if $\AlgA \in \FSIHeyt$ then
\begin{equation}
\eqc(\chi(\AlgA)) = \max\set{ \classV' \in \SubHeyt}{ \AlgA \notin \classV' }, \label{hsplit}
\end{equation}
that is $\eqc(\chi(\AlgA))$ is a largest subvariety of $\Heyt$ not containing algebra $\AlgA$. Indeed, if $\eqc(\chi(\AlgA)) \subset \classV \subseteq \Heyt$ then $\chi(\AlgA)$ is invalid in some algebra $\AlgB \in \classV$ and, by virtue of the Theorem \ref{thrmord}(d), $\AlgA$ is isomorphic with some subalgebra of a homomorphic image of $\AlgB$, therefore, $\AlgA \in \classV$.

Recall from \cite{Day_1973,Day_1975} that an algebra $\AlgA$ is said to be a \textit{splitting algebra in a variety} $\classV$ if there is a largest subvariety of $\classV$ not containing $\AlgA$. The variety generated by a splitting algebra $\AlgA$ we will call a \textit{splitting variety} and the largest variety not containing $\AlgA$ we call a \textit{co-splitting variety}. A splitting variety is generated by a finitely generated s.i. algebra (cf. \cite{McKenzie_1972}), while a co-splitting variety is defined by a formula that we will call a \textit{splitting formula}. Thus we have established the following

\begin{prop} In $\Heyt$  every finite s.i. algebra $\AlgA$ is a splitting algebra and $\chi(\AlgA)$ is a splitting formula. 
\end{prop}

Later we will see that converse statement is true also. But first let us  prove one of the principal results from \cite{Jankov_1969}, namely, let us show that every characteristic formula is $\land$-irreducible.

Recall from \cite{Jankov_1969} that a formulas $A$ is $\land$\textit{-irreducible} (in IPC) if for any formulas $B$ and $C$
\[
(B \land C) \sim A \text{ entails } B \sim A \text{ or } C \sim A 
\] 
and $A$ is called $\land$-\textit{prime }(in IPC) if
\[
B , C \Vdash A \text{ entails } B \Vdash A \text{ or } C \Vdash A. 
\] 

A formula $A$ is \textit{strictly }$\land$\textit{-prime} (comp. \cite{McKenzie_1972}) if for any set of formulas $\Gamma$
\[
\Gamma \Vdash A \text{ entails } B \Vdash A \text{ for some } B \in \Gamma. 
\]

Clearly, any strictly $\land$-prime formulas is $\land$-prime and any $\land$-prime formula is $\land$-irreducible. Recall that relation $\Vdash$ is finitary, that is for any formula $A$ and any set of formulas $\Gamma$
\[
\Gamma \Vdash A \text{ if and only if } \Gamma' \Vdash A \text{ for some finite } \Gamma' \subseteq \Gamma. 
\]
Hence, due to properties of $\land$, if a formula $A$ is $\land$-prime it is strictly $\land$-prime.

Observe that any formula that defines a co-splitting variety is $\land$-prime (cf. \cite{McKenzie_1972}). Indeed, if a formula $A$ is not $\land$-prime one of its $\land$-factors would define a strongly greater variety not containing the splitting algebra and this contradicts the definition of splitting. Thus every characteristic formula is $\land$-prime. The proof of the converse statement is one of the principal results of \cite{Jankov_1969}.

\begin{theorem} (comp. \cite{Jankov_1969}) If $A$ is a $\land$-prime formula and $\nvdash A$, then $A$ is interderivable with some characteristic formula.
\end{theorem}
\begin{proof} Since $\nvdash A$, by virtue of the Proposition \ref{pre-true}, there is a finite algebra $\AlgA$ in which formula $A$ is pre-true. By the Proposition \ref{pre-true1}(b) the algebra $\AlgA$ is s.i. Let us prove that $A \sim \chi(\AlgA)$.

Since $\AlgA \not\models A$, by (Ded) $A \Vdash \chi(\AlgA)$ and we need only to prove $\chi(\AlgA) \Vdash A$.

Let $\Vars_0 = \Vars(A)$ and $\nu_1,\dots,\nu_m$ are all the valuations refuting $A$ in $\AlgA$, that is $\nu_i: \Vars_0 \impl \AlgA$ and $\nu_i(A) \neq \one$ for all $i=1,\dots,m$. Since $A$ is pre-true we can apply the Proposition \ref{pre-true1}(c) and conclude that
\begin{equation}
\nu_i(A) = \omega(\AlgA) ; i=1,\dots,m. \label{omega}
\end{equation}
Now, let us consider the formula
\[
B \bydef \land_{i=1,m}\chi(\AlgA,\Vars_0,\nu_i) \impl A.
\]
Clearly,
\[
\chi(\AlgA,\Vars_0,\nu_1),\dots,\chi(\AlgA,\Vars_0,\nu_m), B \Vdash A.
\]
Recall that $A$ is $\land$-prime, hence $A$ is interderivable either with $\chi(\AlgA,\Vars_0,\nu_i)$ for some $i; 1 \leq i \leq m$, or $A$ is interderivable with $B$. Thus, in order to finish the proof it suffices to demonstrate that $A$ is not interderivable with $B$. Let us show that $\AlgA \models B$ while, as we know, $\AlgA \not\models A$

Proof by contradiction: let us assume that $\AlgA \not\models B$ and $\nu$ is a refuting valuation. Then $\nu(A) \neq \one$ and therefore $\nu = \nu_i$ for some $i, 1 \leq i \leq m$. By \eqref{omega} $\nu_i(A) = \omega(\AlgA)$. On the other had, $\chi(\AlgA,\Vars_0,\nu_i)$ is one of the conjuncts of the premise of $B$. By the definition of characteristic formula 
\[
\nu_i(\chi(\AlgA,\Vars_0,\nu_i)) = \omega(\AlgA) = \nu_i(A).
\]
The latter contradicts that $\nu_i$ is a refuting valuation. 
 \end{proof}

\begin{cor} A formula $A$ is $\land$-prime if and only if it is interderivable with some characteristic formula.
\end{cor}

\subsection{Characteristic Formulas vis a vis Pre-true Formulas} \label{HeytPreTrue}

Let us observe that from the Theorem  \ref{thrmord} it follows that the characteristic formula $\chi(\AlgA)$ of any finite s.i. algebra $\AlgA$ is pre-true in $\AlgA$. Moreover, the following holds.

\begin{prop}
Suppose $\AlgA_1,\dots,\AlgA_n$ is an anti-chain of finite s.i. algebras. Then the formula $A \bydef \wedge_{i=1}^n\chi(\AlgA_i)$ is pre-true in every algebra $\AlgA_i; 1 \leq i \leq n$.  
\end{prop}
\begin{proof} First, let us note that formula $\chi(\AlgA_i)$ is pre-true in $\AlgA_i$. Then, from the definition of anti-chain and by virtue of the Theorem \ref{thrmord},  we can conclude $\AlgA_j \models \chi(\AlgA_i)$ for any $i \neq j; i,j = 1,\dots,n$.
\end{proof}

On the other hand, not every pre-true formula is interderivable with a characteristic formula. 

\begin{prop} Suppose $\AlgA$ is an infinite algebra and a formula $A$ is pre-true in $\AlgA$. Then the formula $A$ is not interderivable with any characteristic formula.
\end{prop}
\begin{proof} For contradiction: assume that $A$ is interderivable with a characteristic formula of some algebra $\AlgB$. Then, since $\AlgA \not\models A$, we have $\AlgA \not\models \chi(\AlgB)$. Hence, by the Theorem \ref{thrmord}, $\AlgB \in \CSub\CHom\AlgA$. The latter means that $\AlgB$ is either a proper subalgebra of $\AlgA$, or a subalgebra of a proper homomorphic image of $\AlgA$, for $\AlgB$ is finite while $\AlgA$ is not. By the definition of pre-true formula, $\AlgB \models A$ and, hence, $\AlgB \models B$ which is not true.
\end{proof}

The examples of infinite algebras that have pre-true formulas can be found in \cite{Kuznetsov_Gerchiu,Wronski_1973}. 

Let us also note that algebra $\Z_\infty + \Z_7 + \Z_2$ from \cite{Kuznetsov_Gerchiu} has a pre-true formula. Clearly, infinite s.i. algebra cannot be finitely presented in $\Heyt$ because $\Heyt$ is finitely approximated. It is natural to ask whether this algebra is finitely presented in some subvariety $\Heyt'$ of $\Heyt$ (which may be not finitely approximated) and, thus, the pre-true formula may be interderivable with a characteristic formula relative to $\Heyt'$. The negative answer to this question was given in \cite[Corollary 2.15]{Citkin_Splitting_2012}: the algebra $\Z_\infty + \Z_7 + \Z_2$ is not finitely presented in any subvariety of $\Heyt$. In the Section \ref{PreTrue} we will discuss the relations between characteristic and pre-true formulas in more details.

\subsection{Characteristic Formulas vis a vis Splitting Formulas} \label{charsplit}

As we already know, the splitting algebras in $\Heyt$ are exactly the finite s.i. algebras and that the finitely presented in $\Heyt$ s.i. algebras are exactly the finite s.i. algebras too. Thus, in $\Heyt$ the sets of splitting and characteristic formulas coincide. But use of the defining formulas instead of diagram formulas makes characteristic formula relative to a variety: an algebra can be not finitely presented in the variety $\Heyt$, but it can be finitely presented in a subvariety $\Heyt' \subset \Heyt$. The similar argument can be applied to a splitting algebra: algebra $\AlgA$ may be not a splitting algebra in $\Heyt$, but it can be a splitting algebra in some subvariety $\Heyt' \subset \Heyt$. Recall also that not every subvariety of $\Heyt$ is finitely approximated. Later we will prove (the Theorem \ref{hom}) for a broad class of varieties (that includes all non-trivial varieties of Heyting algebras) that every finitely presented s.i. algebra is a splitting algebra (which is true for all varieties with EDPC \cite[Corollary 3.2]{Blk_Pgz_1}). This means that in every subvariety $\Heyt' \subseteq \Heyt$ any characteristic (relative to this subvariety) formula is indeed a splitting formula. The converse though is not true: by the Theorem 3.5. from  \cite{Citkin_Splitting_2012} the algebra $\Z_\infty + \Z_7 + \Z_2$ splits the variety $\Heyt'$ that it generates. On the other hand, by \cite[Lemma 2.13 ]{Citkin_Splitting_2012} all finitely presented  in $\Heyt'$ s.i. algebras are finite and belong to the co-splitting variety. Thus, the co-splitting variety of the splitting defined by $\Z_\infty + \Z_7 + \Z_2$ cannot be defined by any characteristic formula and, therefore, this splitting formula is not characteristic.

\subsection{Characteristic Formulas in the Fragments} In \cite{Jankov_1969} V.~Jankov was considering not only intuitionistic logic but also its $\land, \lor, \impl$ and $\land, \impl$ fragments: the logics defined by the axioms of IPC containing just the variables from the respective sets. The simple analysis of definitions and arguments shows that all statements that we had proved hold for these two logics as well. 

In \cite{Jankov_1969} V.~Jankov was also considering the minimal logic and its fragments. The minimal logic is a logic in the signature $\land,\lor,\impl,\neg,\bot$ where $\neg p$ is an abbreviation for $p \impl \bot$ defined by the axioms of IPC not containing $\neg$. The algebraic models for the minimal logic are distributive lattices with relative pseudocomplement and a constant $\bot$. For minimal logic and its fragments we also can repeat all the definitions and proofs.

In \cite{Dziobiak_Finite_1982} the characteristic and pre-true formulas were used in order to construct a continuum of fragments containing $\lor$ and $\impl$ and lacking fmp.

It is worth noticing that the definition of characteristic formula and its properties are relaying on relations between implication and the congruences. This will allow us (in the Section \ref{alg}) to extend the notion of characteristic formula to very broad classes of logics.

\subsection{Further Developments}

Soon after the characteristic formulas were introduced it became clear that they give us a very convenient and powerful tool in studying the intermediate logics or varieties of Heyting algebras. So, it was natural to generalize this notion. First, in the Section \ref{alg} we will discuss how the notion of characteristic formula can be extended to a very broad classes of logics (varieties of algebras). Then, in the Section \ref{inf}, we will see how one can construct characteristic formulas for infinite algebras. And in the Section \ref{axiom} we will learn how characteristic formulas can be used for axiomatization and how an inability to axiomatize every logic by characteristic formulas led to the notion of canonical formula. Lastly, we also show how characteristic formulas can be used as a syntax mean of refutation.         

\section{Characteristic Formulas in the Algebraic Settings} \label{alg}
 
In this section we will generalize the notion of the characteristic formula to varieties with a ternary deductive term (TD term for short) introduced in \cite{Blk_Pgz_3}. As we will see in the Section \ref{secexamples}, the vast majority of logics has as algebraic semantic a variety with a TD term. 

\subsection{Algebraic Semantic} \label{algsemsec}

We consider a \textit{(propositional) language} $\Lang$ consisting of a denumerable set of (propositional) variables $\Vars$ and a finite set of connectives $\Con$. The \textit{formulas} are built in a regular way from the variables from $\Vars$ by using the connectives from $\Con$. The set of all formulas is denoted by $\Frm$.

A \textit{deductive system} is a pair $\DS = \lbr \Lang, \vdash_\DS \rbr$ where $\vdash_\DS$ is a \textit{consequence }relation defined on the subsets of $\Frm$ and formulas and $\vdash_\DS$ satisfies the following conditions: for all $\Gamma, \Delta \subset \Frm$ and $ A \in \Frm$
\begin{itemize}
\item[(i)] $A \in \Gamma$ implies $\Gamma \vdash_\DS A$;
\item[(ii)] $\Gamma \vdash_\DS A$ and $\Gamma \subseteq \Delta$ implies $\Delta \vdash_\DS A$;
\item[(iii)] $\Gamma \vdash_\DS A$ and $\Delta \vdash_\DS B$ for every $B \in \Gamma$ implies $\Delta \vdash_\DS A$;
\item[(iv)] $\Gamma \vdash_\DS A$ implies $\Gamma' \vdash_\DS A$ for some finite $\Gamma' \subseteq \Gamma$;
\item[(v)] $\Gamma \vdash_\DS A$ implies $\sigma(\Gamma) \vdash_\DS \sigma(A)$ for every substitution $\sigma$
\end{itemize}
(here and later by $\sigma(\Gamma)$ we denote the set obtained by applying $\sigma$ to each element of $\Gamma$). From this point forward we assume that some language is given and we will interchangeably use terms "deductive system" and "consequence relation". If $\Gamma \vdash_\DS A$ we say that formula $A$ is \textit{derivable from the formulas} $\Gamma$. The formulas derivable from the empty set we call \textit{theorems} (of the underlaying deductive system $\DS$). And the set of all theorems of $\DS$ we call a \textit{logic of} $\DS$ (and we will omit the reference to $\DS$ when no confusion arises). The logic of a deductive system $\DS$ is denoted by $\LogL(\DS)$ (or $\LogL(\vdash_\DS)$). 

The algebras in the signature $\Con$ (or $\Lang$-\textit{algebras}) we be used as the models, and we are interpreting connectives in an obvious way. We also assume that there is a mapping $\epsilon$ from $\Frm$ into the collection of finite sets of identities and we will call $\epsilon$ a \textit{translation}. A deductive system $\DS$ is said to have a quasivariety $\qvar$ (of $\Lang$-algebras) as an \textit{equivalent algebraic semantic} (comp. \cite{Blk_Pgz_Algebraizable}) if there is such a translation $\epsilon$ that for any $A_1,\dots,A_n,A \in \Frm$
\[
A_1,\dots,A_n \vdash_\DS A \text{ if and only if } \epsilon(A_1),\dots, \epsilon(A_n) \dimpl \epsilon(A) \text{ hold in } \qvar, 
\]
where $\epsilon(A_1),\dots, \epsilon(A_n) \dimpl \epsilon(A) $ is a set of quasi-identities 
\[
 \epsilon(A_1),\dots, \epsilon(A_n) \dimpl \idn ; \idn \in \epsilon(A).  
\]
It is clear that if $\qvar$ is an equivalent algebraic semantic for a deductive system $\DS$, then $A \in \LogL(\DS)$ if and only if the identities $\epsilon(A)$ hold in every algebra of a variety generated by $\qvar$. We will call this variety an \textit{equivalent algebraic semantic} (EAS) for the logic $\LogL = \LogL(\DS)$ and denote it by $\eqc(\LogL))$. For the  sake of simplicity we will be using variables $\Vars$ as object variables, thus formulas also will be terms in algebraic sense, i.e if $A,B$ are formulas then $A \approx B$ is an identity.

\begin{example} The variety of Heyting algebras is an EAS for the $\rm IPL$ (in the signature $\land, \lor, \impl, \neg, \one$) with translation $\epsilon(A) \mapsto A \approx \one$. The variety of interior algebras in the signature $\land, \lor, \impl,\neg, \one,\Box$ with translation $\epsilon(A) \mapsto \Box A \approx \one$ is an EAS for the modal logic $\rm S4$. 
\end{example}

From this point forward we consider only the logics with EAS, and this makes it possible to employ the algebraic means and study the respective varieties of algebras that represent logics rather than logics themselves. 

\subsection{Characteristic Identities in the Varieties with a TD term}

In this section we extend the notion of characteristic formula to the varieties with a TD term.

\subsubsection{Basic Definitions}

If $\AlgA$ is an algebra by $Con(\AlgA)$ we denote a set of all congruences on $\AlgA$ and
by $Con'(\AlgA)$ - the set of all distinct from identity congruences on $\AlgA$.

Let $\AlgA$ be an algebra, $\alga,\algb \in \algA$ and $\alga \neq \algb$.  We will call elements $\alga$ and $\algb$ \textit{indistinguishable} if for any congruence $\theta \in Con'(\AlgA)$
\[
\alga \equiv \algb \pmod{\theta}.
\] 
Clearly, an algebra is s.i. if and only if it has a pair of indistinguishable elements. It is easy to see that the following holds.

\begin{prop} Let $\AlgA$ be an algebra, $\alga,\algb \in \algA$ and $\alga \neq \algb$. Then the following are equivalent:
\begin{itemize}
\item[(a)] elements $\alga, \algb$ are indistinguishable;
\item[(b)] $\alga \not\equiv \algb \pmod{\mu(\AlgA)}$
\item[(c)] $\mu(\AlgA)= \theta(\alga, \algb)$;
\item[(d)] for any homomorphism $\phi$ that is not an isomorphism $\phi(\alga) = \phi(\algb)$. 
\end{itemize}
\end{prop}

Let us also note the following corollary.

\begin{cor} \label{monolith} Let $\AlgA$ be an s.i. algebra, $\alga,\algb$ be indistinguishable elements and $\phi: \AlgA \impl \AlgB$ be a homomorphism of $\AlgA$ in $\AlgB$. If $\phi(\alga) \neq \phi(\algb)$ then $\phi$ is an embedding.  
\end{cor}

If $\AlgA$ is an algebra and $\idn \bydef t(x_1,\dots,x_m) \approx t'(x_1,\dots,x_m)$ is an identity, by $\AlgA \models \idn$ we express the fact that $\idn$ is valid in $\AlgA$, that is, for any $\alga_1,\dots,\alga_m \in \algA$ we have $t(\alga_1,\dots,\alga_m) = t'(\alga_1,\dots,\alga_m)$. If $\classK$ is a class of algebras and $\AlgA \models \idn$ for every $\AlgA \in \classK$ we write $\classK \models \idn$.

Let $\classK$ be a class of algebras and $\idn_1,\dots,\idn_n,\idn$ be identities. The identity $\idn$ is called \cite{MaltsevBook} a $\classK$-\textit{consequence} of $\idn_1,\dots,\idn_n$ (in symbols $\idn_1,\dots,\idn_n \vdash_\classK \idn$) if a quasi-identity $\idn_1,\dots,\idn_n \dimpl \idn$ holds in $\classK$. Identities $\idn_1$ and $\idn_2$ are $\classK$-\textsl{equivalent} if $\idn_1 \vdash_\classK \idn_2$ and $\idn_2 \vdash_\classK \idn_1$ (in symbols $\idn_1 \sim_\classK \idn_2$). We also say that $\idn$ $\classK$\textit{-follows} from $\idn_1,\dots,\idn_n$ (in symbols $\idn_1,\dots,\idn_n \vDash_\classK \idn_2$) if $\AlgA \models \idn$ as long as $\AlgA \models \idn_i$ for all $i=1,\dots,n$ and for every $\AlgA \in \classK$ . If $\idn_1 \vDash_\classK \idn_2$ and $\idn_2 \vDash_\classK \idn_1$  we say that identities $\idn_1$ and $\idn_2$ are $\classK$\textit{-equipotent} (in symbols $\idn_1 \approx_\classK \idn_2$). 
Clearly, if $\classV$ is a variety, then two identities are $\classV$-equipotent if and only if these identities define in $\classV$ the same subvarieties. 
It is easily seen that  
\begin{equation}
\idn_1 \sim_\classK \idn_2 \dimpl \idn_1 \approx_\classK \idn_2 \label{equiv}
\end{equation}
If $\Gamma_1,\Gamma2$ are sets of identities we say that $\Gamma_1$ and $\Gamma_2$ are \textit{equivalent} (in symbols $\Gamma_1 \sim_\classK \Gamma2$) if $\Gamma_1 \vdash_\classK \idn$ for every $\idn \in \Gamma_2$ and $\Gamma_2 \vdash_\classK \idn$ for every $\idn \in \Gamma_1$. If one of the sets, let say $\Gamma_2$, consists of a single identity $\idn$ we will omit braces and will write $\Gamma_1 \sim \idn$. In a similar way we will be using $\Gamma_1 \approx_\classK \Gamma_2$.

\subsubsection{TD Term}

Let $\classV$ be a variety of algebras in the finite signature $\Sign$. A ternary term $td(x,y,z)$ is called \cite{Blk_Pgz_3} a \textit{ternary deductive term of variety} $\classV$ (a TD term) if for any algebra $\AlgA \in \classV$ and any elements $\alga,\algb,\algc,\algd$
\begin{equation}
\begin{split}
 & td(\alga, \alga, \algb) = \algb,\\
 &  td(\alga, \algb, \algc) = td(\alga, \algb, \algd) \text{ if } \algc \equiv \algd  \pmod{\theta(\alga, \algb)}.
\end{split} \label{td}
\end{equation}

As we can see, a TD term gives us a uniform way to define the principal congruences. By iterating the TD term a \eqref{td}-like characterization of principal congruences can be extended to the compact congruences (cf. \cite[Theorem 2.6]{Blk_Pgz_3}): if $\ua,\ub$ are lists of elements of an algebra $\AlgA \in \classV$ then 
\begin{equation}
\algc \equiv \algd) \pmod{\theta(\ua, \ub)} \text{ iff }  td(\ua,\ub,\algc) = td(\ua,\ub,\algd), \label{adpc}
\end{equation}
where 
\begin{equation}
td(\ua,\ub,c) = td(\alga_1,\algb_1,td(\alga_2,\algb_2,\dots td(\alga_m,\algb_m,\algc))\dots). \label{tdl}
\end{equation}
Using a simple induction it is not hard to prove that for any $\ua \subseteq \algA$ and any $\algc \in \algA$
\begin{equation} 
td(\ua,\ua,\algc) = \algc. \label{td1}
\end{equation}

\subsubsection{Finitely Presented Algebras}

Assume $\ux \bydef x_1,\dots,x_n$ is a list of variables and 
\[
D \bydef \lbrace t_1(\ux) \approx t_1'(\ux),\dots,t_m(\ux) \approx t_m'(\ux) \rbrace
\]
is a set of equalities. Then the set $D$ defines in a given variety $\classV$ a unique (modulo isomorphism) algebra $\AlgA$ for which the following hold:
\begin{itemize}
\item[(a)] algebra $\AlgA$ is generated by some set of elements $\alga_1,\dots,\alga_n$; 
\item[(b)] $t_1(\ua) = t_1'(\ua),\dots, t_m(\ua) = t_m'(\ua)$;
\item[(c)] if for elements $\algb_1,\dots,\algb_m$ of some algebra $\AlgB \in \classV$ we have 
\[
t_1(\ub) = t_1'(\ub),\dots, t_m(\ub) = t_m'(\ub)
\] 
then the mapping $\phi: \alga_i \mapsto \algb_i, i=1,\dots,m$ can be extended to homomorphism $\overline{\phi}: \AlgA \impl \AlgB$.
\end{itemize}
The set $D$ is called the set of \textit{defining relations} and algebra $\AlgA$ is called \textit{finitely presented in variety} $\classV$. If $D$ is a set of defining relations for $\AlgA$ we will express this fact by $\AlgA_\classV(D)$.

\subsubsection{Characteristic Identities: Definition}
Let $\classV$ be a variety and $\AlgA_\classV(D)$ be an s.i. algebra and $\algb_1,\algb_2 \in \AlgA(D)$ be indistinguishable elements.
Since $\alga_1,\dots,\alga_2 \in \algA$ is a set of generators for some terms $r_1(\ux),r_2(\ux)$ we have
\[
\algb_1 = r_1(\ua) \text{ and } \algb_2 = r_2(\ua).
\]

Assume that a variety $\classV$ has a TD term $td(x,y,z)$ and $\AlgA$ is an s.i. finitely presented algebra with defining relations $D$ and $r_1,r_2$ are terms expressing a pair of indistinguishable elements. Let $\ut(\ux) \bydef t_1(\ux),\dots,t_m(\ux)$ and  $\ut'(\ux) \bydef t_1'(\ux),\dots,t_m'(\ux)$. Then the identity
\begin{equation}
 td(\ut(\ux),\ut'(\ux),r_1(\ux)) \approx td(\ut(\ux),\ut'(\ux),r_2(\ux)) \label{chari}
\end{equation}
will be called a \textit{characteristic identity of algebra} $\AlgA$ and we will denote this identity by $\chi_{_\classV}(\AlgA,D,r_1,r_2)$ (or simply by $\chi_{_\classV}(\AlgA)$ or $\chi(\AlgA)$ if no confusion arises).

\subsubsection{Main Theorem}

The theorem and the corollaries of this section extend the results by V.~Jankov (cf. \cite{Jankov_1963_ded,Jankov_1969} and the Section \ref{Jankovprop}) to varieties with a TD term.

\begin{theorem} \label{hom} (comp. \cite{Jankov_1969})If a variety $\classV$ has a TD term $td(x,y,z)$ and $\AlgA$ is an s.i. finitely presented in $\classV$ algebra then for every $\AlgB \in \classV$ 
\[
\AlgB \not\models \chi_{_\classV}(\AlgA,D,r_1,r_2) \text{ if and only if } \AlgA \in \CSub\CHom\AlgB.
\] 
\end{theorem}
\begin{proof}
First we prove that if $\AlgB \not\models \chi_{_\classV}(\AlgA,D,r_1,r_2)$ then $\AlgA \in \CSub\CHom\AlgB$.

Let $\AlgB \in \classV$ and $\ub \bydef \algb_1,\dots,\algb_m$ be elements of $\AlgB$ refuting $\chi(\AlgA,D,r_1,r_2)$, that is
\begin{equation}
td(\ut(\ub),\ut'(\ub),r_1(\ub)) \neq td(\ut(\ub),\ut'(\ub),r_2(\ub)). \label{refut} 
\end{equation} 
Let $\uc \bydef t_1(\ub),\dots, t_m(\ub)$ and $\uc' \bydef t_1'(\ub),\dots, t_m'(\ub)$. Then from \eqref{refut}, by virtue of \eqref{adpc}, we can conclude that
\begin{equation}
r_1(\ub) \not\equiv r_1(\ub) \pmod{\theta(\uc,\uc')}. \label{refut1}
\end{equation}
Let $\phi: \AlgB \impl \AlgB/\theta(\uc,\uc')$ be a natural homomorphism of $\AlgB$ onto quotient algebra $\AlgB/\theta(\uc,\uc')$. Recall that the congruence $\theta(\uc,\uc')$ is generated by pairs $(t_i(\ub),t_i'(\ub)); i=1,\dots,m$ and, hence, for every $i=1,\dots,m$ we have $t_i(\ub) \equiv t_i'(\ub) \pmod{\theta(\uc,\uc')}$. Therefore for every $i=1,\dots,m$ we have 
\[
\phi(t_i(\ub)) = \phi(t_i'(\ub)
\]
and, since $\phi$ is a homomorphism, 
\[
t_i(\phi(\ub)) = t_i'(\phi(\ub)).
\]
 
Let $\algb_i' = \phi(\algb_i); i=1,\dots,n$. Then in $\AlgB/\theta(\uc,\uc')$
\begin{equation}
t_i(\ub') = t_i'(\ub') \text{ for all } i=1,\dots,m,  \label{defh}
\end{equation}
while from \eqref{refut1}
\begin{equation}
r_1(\ub') \neq r_2(\ub').  \label{monol2}
\end{equation}

Suppose $\alga_1,\dots,\alga_n$ are the generators from the definition of the finitely presented algebra $\AlgA$. And let us consider the mapping
\begin{equation}
\psi: \alga_i \mapsto \algb'_i; i=1,\dots,n. \label{map}
\end{equation}
By \eqref{defh} and the definition of finitely presented algebra, we can extend mapping $\psi$ to a homomorphism $\opsi: \AlgA \impl \AlgB/\theta(\uc,\uc')$. All what is left to prove is that $\opsi$ is an isomorphism, i.e. that $\psi$ is an embedding.

Indeed, from \eqref{monol2}
\[
\opsi(r_1(\ub)) \neq \opsi(r_2(\ub)).
\]
Recall that $r_1(\ub)$ and $r_1(\ub)$ are indistinguishable elements. Thus, we can apply the Corollary \ref{monolith}  and conclude that $\psi$ is an embedding.

Conversely, assume that $\AlgA \in \CSub\CHom\AlgB$. In order to prove $\AlgB \not\models \chi_{_\classV}(\AlgA,D,r_1,r_2)$ it suffices to verify that $\AlgA \not\models \chi_{_\classV}(\AlgA,D,r_1,r_2)$. 

Indeed, by the definition of finitely presented algebra, 
\[
t_i(\ua) = t_i'(\ua) \text{ for all } i=1,\dots,m.
\] Hence, if $\algb_i = t_i(\ua); i=1,\dots,m$ and $\ub \bydef \algb_1,\dots,\algb_m$, then
\begin{equation*}
\chi_{_\classV}(\AlgA,D,r_1,r_2)(\ua)  
\end{equation*} 
is true if and only if
\begin{equation*}
td(\ub,\ub,r_1(\ua)) = td(\ub,\ub,r_2(\ua)).
\end{equation*}
But by \eqref{td1}
\[
td(\ub,\ub,r_1(\ua)) = r_1(\ua) \text{ and } td(\ub,\ub,r_1(\ua)) = r_2(\ua).
\]
Since, by the hypothesis of the theorem, $r_1(\ua)$ and $r_2(\ua)$ are distinct elements, we can conclude that the the valuation $\nu: \ux \mapsto \ua$ refutes the characteristic identity $\chi_{_\classV}(\AlgA,D,r_1,r_2)$.
\end{proof}


\begin{cor} \label{genrefute} Suppose $\AlgA$ is an s.i. algebra finitely presented in a variety $\classV$ and $\chi_{_\classV}(\AlgA,D,r_1,r_2)$ is a characteristic identity. If $\idn(y_1,\dots,y_s)$ is an identity and $\AlgA \not\models \idn$ then there is such a substitution $\sigma$ that $\sigma(\idn) \vdash_\classV \chi_{_\classV}(\AlgA,D,r_1,r_2)$.
\end{cor}
\begin{proof} Since $\AlgA \not\models \idn$, there are elements $\algc_1,\dots,\algc_s$ such that $\idn(\algc_1,\dots,\algc_s)$ is not true. Assume $\alga_1,\dots,\alga_n$ are the generators of algebra $\AlgA$ from the definition of finitely presented algebra. Thus, we can express every element $\algb_i$ via generators, that is $\algc_i = t_i(\alga_1,\dots,\alga_n); i=1,\dots,s$ for some terms $t_i$. Let us consider the substitution
\[
\sigma: y_i \mapsto t_i(x_1,\dots,x_n); i=1,\dots,s.
\]  
It is clear that $\sigma(\idn)(\alga_1,\dots,\alga_n)$ is not true.

Next, we want to demonstrate that if for any elements $\algb_1,\dots,\algb_n$ of any algebra $\AlgB \in \classV$ we have $\chi_{_\classV}(\AlgA,D,r_1,r_2)(\algb_1,\dots,\algb_n)$ is not true, then $\sigma(\idn)(\algb_1,\dots,\algb_n)$ is not true too.

  Indeed, if $\chi_{_\classV}(\AlgA,D,r_1,r_2)$ is invalid in an algebra $\AlgB \in \classV$ then, by virtue of Theorem \ref{hom}, we have that the mapping $\psi$ defined by \eqref{map} can be extended to an isomorphism $\opsi$. Hence, $\sigma(\idn)(\algb_1',\dots,\algb_n')$ is not true, for $\sigma(\idn)(\alga_1,\dots,\alga_n)$ is not true. Recall from the proof of the Theorem \ref{hom} that elements $\algb_i'$ are homomorphic images of elements $\algb_i$ (on which the characteristic identity got refuted). Thus, $\sigma(\idn)(\algb_1,\dots,\algb_n)$ cannot be true.      	
\end{proof}

\begin{cor} \label{meetpr} If $\AlgA$ is an s.i. algebra finitely presented in a variety $\classV$, then all characteristic identities of $\AlgA$ are $\classV$-equipotent.
\end{cor}

Given a variety $\classV$, an identity $\idn$ is called $\land$\textit{-prime in} $\classV$ if for any set of identities $\classI$ if $\classI \models_\classV \idn$ then $\idn' \models_\classV \idn$ for some $\idn' \in \classI$. Often, if no confusion arises, we will omit the reference to a variety.

\begin{cor} \label{meet} Any characteristic identity is $\land$-prime.
\end{cor}
\begin{proof} Let $\chi$ be a characteristic identity of an algebra $\AlgA$ and $\classI$ be a class of identities such that $\classI \vDash \chi$. Since $\AlgA \not\models \chi$ there is such an identity $\idn \in \classI$ that $\AlgA \not\models \idn$. By virtue of the Corollary \ref{genrefute}, $\idn\vDash \chi$. 
\end{proof}

\subsection{Examples} \label{secexamples}

Let $\Heyt$ be a variety of Heyting algebras in the signature $\land,\lor,\impl,\neg,\one$. Variety $\Heyt$ has two TD terms \cite{Blk_Pgz_3}:
\[
\begin{split}
& td_\impl(x,y,z) \bydef ((x \impl y) \land (y \impl x)) \impl z \text{ and}\\
& td_\land(x,y,z) \bydef ((x \impl y) \land (y \impl x)) \land z
\end{split}
\] 

For Heyting algebras we always can assume that finitely presented algebra is defined by a single identity of the form $t(\ux) \approx \one$. Also, every s.i. Heyting algebra has an opremum which always is a member of the monolith. Therefore we always can take a term $r(\ux)$ that expresses the opremum and $\one$ as $r_1$ and $r_2$. Thus, using $td_\impl$ we will get the following characteristic term
\[
((t(\ux) \impl \one) \land (\one \impl t(\ux))) \impl r(\ux) \approx ((t(\ux) \impl \one) \land (\one \impl t(\ux))) \impl \one
\]
which clearly is equivalent to
\begin{equation}
	t(\ux) \impl r(\ux) \approx \one. \label{tdimpl}
\end{equation}
If we take as a defining relations the diagram $\delta^+(\AlgA)$ of an algebra $\AlgA$, we will get the identity corresponding to the Jankov formula. 

Note that if we use a different TD term, for instance, $td_\land$, we will get
\[
((t(\ux) \impl \one) \land (\one \impl t(\ux))) \land r(\ux) \approx ((t(\ux) \impl \one) \land (\one \impl t(\ux))) \land \one,
\]
which is equivalent to
\begin{equation}
	t(\ux) \land r(\ux) \approx t(\ux).  \label{tdland}
\end{equation}
It is easy to see that the identities \eqref{tdimpl} and \eqref{tdland} are $\Heyt$-equipotent. 

\begin{remark} In \cite{Kowalski_Ono_2000,Kowalski_Ono_Remarks_2000} for constructing a characteristic term the greatest and the smallest elements of a monolith are used. In fact, one can use any two distinct elements of the monolith.

\end{remark}
\subsection{Applications to Different Classes of Logics}

In this section we will discuss how the characteristic formulas can be used in different classes of logics.

\subsubsection{Extensions by Compatible Operations}

Suppose that we have a variety $\classV$ of algebras in the signature $\Con$ and a TD term $td$. Let us recall from \cite{Blk_Pgz_4} that an operation $f$ on algebras from $\classV$ is called \textit{compatible} if extending $\Con$ by adding $f$ to it does not change the congruences. Observe, that if $\AlgA$ is an algebra in the extended signature and $\AlgA^-$ is its $\Con$-reduct, then $\AlgA$ is s.i. if and only if $\AlgA^-$ is an s.i. algebra. Note also that if we endowed the signature by a compatible operation the TD term remains the same. And, hence, we can easily extend the definition of Jankov formula (by adding to the diagram new members representing the new operation) to the new class of algebras. It is also not hard to see that we can extend the notion of characteristic identity. For instance, for Brouwerian semilattices the TD term is
\[
td = (p \impl q) \land (q \impl p) \impl r 
\]  
and Jankov formula of a finite s.i. Brouwerian semilattice $\AlgA$ is
\[
\delta^+(\AlgA) \impl p_\omega.
\] 
We can add to the signature a new symbol $\lor$ and convert Brouwerian semilattices into Brouwerian lattices. $\lor$ is compatible with congruences of Brouverian semilattices. Thus the TD term remains the same and in the Jankov formula the antecedent will contain more conjuncts. Similarly we can add a new constant $0$ and convert Brouwerian lattices into Johansson's algebras (cf. \cite{Odintsov_Structure_2005,Odintsov_Lattice_2006}).   

\subsubsection{Examples of Varieties with a TD Term}

In the following table we give some examples of varieties and their td-terms. For the detail we refer the reader to \cite{Blk_Pgz_3,Blok_Raftery_Varieties_1997,Agliano_Ternary_1998,Berman_Blok_Free_2004,Odintsov_Structure_2005,Spinks_Veroff_Characterisation_2007}. 

We start with noting that all discriminator variueties have a TD term. 

Consider the following ternary terms:

\begin{itemize}
\item[(a)] $td \bydef (p \impl q) \impl ((q \impl p) \impl r)$ 
\item[(b)] $td \bydef (p \impl q) \land (q \impl p) \land r$, or $p \eqv q \land r$
\item[(c)] $td \bydef (p \impl q) \impln{n-1} ((q \impl p) \impln{n-1} r)$
\item[(d)] $td \bydef (p \eqv  q) \land \Box(p \eqv  q) \land \dots \land \Box^{n-1}(p \eqv  q) \impl r$
\item[(e)] $td \bydef (p \eqv  q) \land \Box(p \eqv  q) \land \dots \land \Box^{n-1}(p \eqv  q) \land r$
\item[(f)] $td \bydef \Box(p \eqv  q)  \impl r$
\item[(g)] $td \bydef \Box(p \eqv  q)  \land r$
\item[(h)] $td \bydef (p \eqv  q)^n \cdot r$
\item[(i)] $td \bydef (p \eqv  q) \impln{n} r$
\end{itemize}
where $p \eqv q$ as an abbreviation for $(p \impl q) \land (q \impl p)$. 

\pagebreak 

\begin{table}[ht]
\centering
\begin{tabular}{|l|l|}
\hline
Variety & td-term  \\ 
\hline
Hilbert algebras & (a) \\
Brouwerian semilattices & (a) \\
Brouwerian semilattices & (b)  \\
Heyting algebras & (a) \\
Heyting algebras & (b) \\
KM algebras & (a) \\
KM algebras & (b) \\
Johansson's algebras & (a) \\
Johansson's algebras & (b) \\
n-potent hoops & (c) \\
n-transitive modal algebras & (d) \\
n-transitive modal algebras  & (e) \\
Interior algebras  & (f) \\
Interior algebras  & (g) \\
BCI monoids &  (h) \\
BCI monoids & (i) \\
\hline
\end{tabular}\\

\caption{Examples of TD terms}
\label{tab:myfirsttable}
\end{table}

\section{Characteristic Identities as Means for Axiomatization} \label{axiom}

One of the important properties of logics or varieties is whether a logic, or a variety, admits an independent axiomatization. The characteristic formulas give a very convenient tool for constructing an independent axiomatization of a given logic (by using the characteristic formulas of members of an anti-chain of finite s.i. algebras). In algebraic terms, for a given variety $\classV$ one can try to construct an independent set of identities defining $\classV$, that is such a set $\mathscr{I}$ that the identities $\mathscr{I}$ define $\classV$ and no identity from $\mathscr{I}$ is a consequence from the rest of identities from $\mathscr{I}$ (see, for instance, \cite{Wojtylak_Indep_1989,Budkin_Independent_1994,Sapir_Finite_1980,Gorbunov_Covering_1977}).

\subsection{Logics Axiomatizable by Characteristic Identities}

We recall from \cite{MaltsevBook} that the \textit{rank} of an identity is a number of distinct variables occurring in it. The \textit{axiomatic rank} of a variety $\classV$ is the least natural number $r$ (that we denote by $r_a(\classV)$) such that $\classV$ can be defined (axiomatized) by identities of the rank not exceeding $r$. In this section we show that in locally finite varieties any subvariety $\classV$ can be defined by independent set of characteristic identities of rank not exceeding $r_a(\classV)$. In this section we will study when a logic or a variety admits an axiomatization by characteristic identities, and when it does not admit such an axiomatization.

\subsubsection{Edge Algebras}

In this section we introduce the notion of \textit{edge algebra}\footnote{In \cite[Definition 3.9]{Tomaszewski_PhD} E.~Tomaszewski introduced a notion of critical algebra which is very similar to the notion of edge algebra. We prefer the term ``edge algebra'' since the term ``critical algebra'' is already used in the different meanings.} which is very similar to the notion of an algebra having a pre-true formula. 

\begin{definition}
Let $\classV$ be a variety and $\AlgA$ be an algebra. We will say that an algebra $\AlgA$ is an \textit{edge algebra of} $\classV$ (or $\classV$-\textit{edge}) if $\AlgA \notin \classV$ while $\oV(\AlgA) \subseteq \classV$.
\end{definition}

From the Proposition \ref{pretr}(a) it follows that any $\classV$-edge algebra $\AlgA$ has pre-true identities. In fact, any identity that separates $\AlgA$ from $\classV$ is pre-true in $\AlgA$. From the Proposition \ref{pretr} we also know that any $\classV$-edge algebra is s.i. and finitely generated (cf. \cite[Proposition 3.8]{Tomaszewski_PhD}).

Recall from \cite{MaltsevBook} that the \textit{basis rank} of an algebra $\AlgA$ is a minimal cardinality of a set of the elements generating $\AlgA$. The basis rank of an algebra $\AlgA$ we denote by $r_b(\AlgA)$ and any set of generators of cardinality $r(\AlgA)$ we will call a \textit{basis} (of $\AlgA$). 

The importance of $\classV$-edge algebras can be seen from the following simple proposition.

\begin{prop} \label{edgeax} Let $\AlgA$ is a $\classV$-Edge algebra. Then
\[
r_b(\AlgA) \leq r_a(\classV).
\]
\end{prop}
\begin{proof} Let $\AlgA$ be a $\classV$-edge algebra. Hence, it is finitely generated and it has a finite basis rank. Let $r \bydef r_b(\AlgA)$. For contradiction: assume that $r_a(\classV) < r_b(\AlgA) = r$. Then there is a collection of  identities of the rank less than $r$ that defines the variety $\classV$. Recall that $\AlgA \notin \classV$. Therefore there is an identity $\idn(x_1,\dots,x_m)$ where $m < r$ that is valid in $\classV$ but is not valid in $\AlgA$. Assume $\alga_1,\dots,\alga_m \in \AlgA$ and $\idn(\alga_1,\dots,\alga_m)$ is not true. Let $\AlgA'$ be a subalgebra of $\AlgA$ generated by the elements $\alga_1,\dots,\alga_m$. Since $m < r_b(\AlgA)$, by virtue of the definition of basis rank, $\AlgA'$ is a proper subalgebra of $\AlgA$. And, by the definition of edge algebra, $\AlgA' \in \classV$. Hence, $\idn$ is invalid in $\AlgA'$ and this contradicts that $\idn$ is valid in $\classV$.   
\end{proof}

\begin{example} Let $\AlgA$ be a Heyting algebra of $n$-th slice, that is $\AlgA$ contains a subalgebra isomorphic to $n$-element chain (i.e. linearly ordered) algebra $\AlgC_n$, while $\AlgC_{n+1}$ is not embeddable in $\AlgA$. Since all the subalgebras and homomorphic images of a chain algebra are chain algebras, we can conclude that $\AlgC_{n+1}$ is $\eqc(\AlgA)$-edge algebra. Thus $r_b(\AlgC_{n+1}) \leq r_a(\eqc(\AlgA))$. Observe, that $r_b(\AlgC_n) = 1$ when $n=2$ and $r_b(\AlgC_n) = n-2$, when $n > 2$. Hence, $r_a(\eqc(\AlgA)) \geq n-2$ (cf. \cite[Theorem 2.2]{Bellissima_Finite_1988}).   
\end{example}

\begin{cor} If $\classV$ is a variety and for any natural number $m$ there is a $\classV$-edge algebra of basis rank exceeding $m$, then $\classV$ has an infinite axiomatic rank and, hence, cannot be finitely axiomatized.
\end{cor}

For example (cf. \cite{Maksimova_Skvortsov_Shehtman}), the variety of Heyting algebras corresponding to Medvedev's Logic has the edge algebras of basis rank exceeding any given natural number and, thus, this variety is not finitely axiomatizable.

\subsubsection{Optimal Axiomatization}

Every variety can be defined by a set of identities and every defining set we will call a set of axiom. But not every variety admits an \textit{independent} set of axioms, that is, the set of axiom neither member of which follows from the rest of identities. For instance, even not every variety of Heyting algebras has an independent axiomatization \cite{Chagrov_Zakh_1995_indep}. Also, as we saw, not every variety has a finite axiomatic rank. But if a variety admits a finite axiomatization we can try to find the optimal one. In the locally finite varieties the optimal axiomatization (i.e. containing the least possible number of variables) can be constructed of characteristic formulas. In this section we study the optimal axiomatizations of subvarities of locally finite varieties that have a TD term.

First, let us consider the axiomatization of subvarieties of a given variety $\classV_0$. All the varieties under considerations are the subvarieties of $\classV_0$ and axiomtization is understood as axiomatization over $\classV_0$.

\begin{definition} Let $\classV$ be a variety axiomatized by a set of axioms $\Ax$. Then we will call the set $\Ax$ \textit{optimal} if 
\begin{itemize}
\item[(a)] $\Ax$ is independent;
\item[(b)] Every identity in $\Ax$ is $\land$-prime;
\item[(c)] Axioms of $\Ax$ contain no more than $r_a(\classV)$ distinct variables.
\end{itemize}
\end{definition}

Let us observe that independent axiomatization by $\land$-prime formulas is unique in the following sense.

\begin{prop}
Let $\Ax_1$ and $\Ax_2$ be two independent axiomatizations of a given variety $\classV$ consisting of $\land$-prime identities. Then $|\Ax_n| =|\Ax_m|$ and there is 1-1-correspondence $\phi$ between $\Ax_1$ and $\Ax_2$ such that $\idn$ and $\phi(\idn)$ are $\classV$-equipotent for all $\idn \in \Ax_1$.
\end{prop}
\begin{proof}
Let $\idn \in \Ax_1$. Then, since $\Ax_2$ is a set of axiom of $\classV$, we have 
\[\Ax_2 \models_\classV \idn.\]
Recall that $\idn$ is $\land$-prime. Hence, there is such an identity $\idt \in \Ax_2$ that 
\[
\idt \models_\classV \idn.
\] 
On the other hand, $\Ax_1 \models_\classV \idt$ and, hence, for some $\idn' \in \Ax_1$ we have 
\[
\idn' \models_\classV \idt.
\]
Since $\Ax_1$ is an independent set we can conclude that
\[
\idn = \idn'.
\]
Thus, we have constructed a 1-1-correspondence $\phi: \idn \mapsto \idn'$ between $\Ax_1$ and a subset of $\Ax_2$ and $\idn \approx_\classV \phi(\idn)$. If we take into account that $\Ax_1$ defines $\classV$ and $\Ax_2$ is an independent set, we can conclude that $\phi$ maps $\Ax_1$ onto $\Ax_2$. 
\end{proof}

\subsubsection{Optimal Axiomatization in Locally Finite Varieties}

In this section we will show that given a locally finite variety $\classV_0$ with a TD term any subvariety $\classV \subseteq \classV_0$ admits an optimal axiomatization over $\classV_0$ and we will use the characteristic identities in order to construct such an axiomatization (comp. \cite[Remark 2.8]{Bezhanishvili_N_Characterization_2004}). Till the end of this section we assume that a variety $\classV_0$ is given and we will construct a set of axiom $\Ax$ that defines $\classV$ over $\classV_0$, that is for any $\AlgA \in \classV_0$, $\AlgA \in \classV$ if and only if $\AlgA \models \idn$ for all $\idn \in \Ax$. 

Let us recall from \cite{MaltsevBook} that a variety $\classV$ is called \textit{locally finite} if every finitely generated algebra from $\classV$ is finite. Since every $n$-generated algebra in $\classV$ is a homomorphic image of a free in $\classV$ algebra $\AlgF_\classV(n)$ of rank $n$, the variety $\classV$ is locally finite if and only if all its free algebras of finite rank are finite. 

Since characteristic formulas enjoy the same properties as Jankov formulas we can use the regular technique of constructing axiomatization consisting of Jankov formulas (see, for instance, \cite{Jankov_1969,Citkin1986,Skvortsov_Remark_1999,Tomaszewski_Algorithm_2002,Bezhanishvili_N_PhD}). 

By $FSI(\classV)$ we denote a set of all finite s.i. algebras from $\classV$. Let us introduce on $FSI(\classV_0)$ a partial order:
\begin{equation}
\AlgA \leq \AlgB \text{ if and only if } \AlgA \in \eqc(\AlgB). \label{po}
\end{equation}
The reflexivity and transitivity are trivial. Let us show that $\AlgA \leq \AlgB$ and $\AlgB \leq \AlgA$ yields $\AlgA \cong \AlgB$.
Indeed, if $\AlgA \in \eqc(\AlgB)$ we have $\AlgB \not\models \chi(\AlgA)$ and, by virtue of the Theorem \ref{hom}, $\AlgA \in \CSub\CHom\AlgB$. Hence, $|\algA| \leq |\algB|$. Similarly, $|\algB| \leq |\algA|$. Since algebras $\AlgA$ and $\AlgB$ are finite, we can conclude that $|\algA| = |\algB|$ and these algebras are isomorphic.

\begin{theorem} \label{optax} Let $\classV_0$ be a locally finite variety with a TD term. Then every subvariety $\classV \subsetneq \classV_0$ admits an optimal axiomatization constructed of characteristic identities.
\end{theorem}
\begin{proof}
In order to prove the theorem we will demonstrate that 
\begin{enumerate}
\item The set $FSI(\classV_0 \setminus \classV$) has a non-empty set of minimal elements $MSI(\classV)$ and each algebra from $FSI(\classV_0 \setminus \classV)$ is comparable with some algebra from $MSI(\classV)$;
\item The set $\set{\chi(\AlgA)}{ \AlgA \in MSI(\classV)}$ is an independent axiomatization consisting of $\land$-prime identities;
\item $r_a(\classV) = \max \set{r_b(\AlgA)}{\AlgA \in MSI(\classV)}$.   
\end{enumerate}

(1) First, observe that, since $\classV \subsetneq \classV_0$, there must be a finitely generated s.i. algebra $\AlgA \in \classV_0 \setminus \classV$. The variety $\classV_0$ is locally finite, hence, by the definition, $\AlgA$ is finite and belongs to $FSI(\classV_0 \setminus \classV)$. We also saw that $\AlgA \leq \AlgB$ yields $|\algA| \leq |\algB|$. Hence, there is a minimal (w.r.t. $\leq$) algebra $\AlgA' \leq \AlgA$ and $\AlgA' \in FSI(\classV_0 \setminus \classV)$. Moreover, for each algebra $\AlgA \in FSI(\classV_0 \setminus \classV)$ there is an algebra $\AlgA' \in MSI(\classV)$ such that $\AlgA' \leq \AlgA$.

(2) In order to prove (2) first we need to demonstrate that $\Ax \bydef \set{ \chi(\AlgA)}{ \AlgA \in MSI(\classV)}$ is a set of axioms, that is we need to prove
\begin{itemize}
\item[(a)] every identity $\idn \in \Ax$ is valid in any algebra $\AlgB \in \classV$;
\item[(b)] for every $\AlgC \in FSI(\classV_0 \setminus \classV)$ there is such an identity $\idn \in \Ax$ that $\AlgC \not\models \idn$.
\end{itemize}

(a) Assume for contradiction that $\idn \in \Ax$ and $\AlgB \not\models \idn$ for some algebra $\AlgB \in \classV$. Recall that $\idn$ is a characteristic identity for an algebra from $\classV_0 \setminus \classV$, i.e. $\idn = \chi(\AlgA)$. Thus, $\AlgB \not\models \chi(\AlgA)$ and, by virtue of the Theorem \ref{hom}, $ \AlgA \in \CSub\CHom\AlgB$. Hence, $\AlgA \in \classV$ and this contradicts that $\AlgA \in \classV_0 \setminus \classV$.

(b) Let $\AlgC \in FSI(\classV_0 \setminus \classV)$. As we proved above, there is an algebra $\AlgA \in MSI(\classV)$ such that $\AlgA \leq \AlgC$. Thus, $\chi(\AlgA) \in \Ax$. Recall that $\AlgA \not\models \chi(\AlgA)$. Since $\AlgA \leq \AlgC$, i.e. $\AlgA \in \eqc(\AlgC)$, we can conclude $\AlgC \not\models \chi(\AlgA)$. 

So, we have proved that $\Ax$ is a complete set of axioms defining $\classV$ over $\classV_0$. We also know from the Corollary \ref{meet} that every characteristic identity is $\land$-prime. We need only to verify that $\Ax$ is an independent set. 

Indeed, assume for contradiction that $\Ax' \models_{\classV_0} \idn$ where $\idn \in \Ax$ and $\Ax' = \Ax \setminus \lbrace \idn \rbrace$. Recall that $\idn = \chi(\AlgA)$ for some $\AlgA \in MSI(\classV)$. Hence, since $\AlgA \not\models \idn$ and by the assumption $\Ax' \models \idn$, there must be $\idt \in \Ax'$ such that $\AlgA \not\models \idt$. But, by the definition, $\idt = \chi(\AlgB)$ for some $\AlgB \in MSI(\classV)$. Therefore,
\[
\AlgA \not\models \chi(\AlgB)
\]
and, by virtue of the Theorem \ref{hom},
\[
\AlgB \in \CSub\CHom\AlgA.
\]
Thus, $\AlgB \in \classV(\AlgA)$, i.e. $\AlgB \leq \AlgA$ and this contradicts that $\AlgA \in MSI(\classV)$. 

(3) First, let us observe that for every algebra $\AlgA$ from $MSI(\classV)$ all proper subalgebras and homomorphic images of $\AlgA$ belong to $\classV$, otherwise $\AlgA$ would not be minimal in $FSI(\classV_0 \setminus \classV)$. That is, $\AlgA$ is a $\classV$-edge algebra and, by virtue of Proposition \ref{edgeax},
\[
 \max\lbrace r_b(\AlgA): \AlgA \in MSI(\classV)\rbrace \leq r_a(\classV). 
\] 
Thus, all we need to do is while constructing $\Ax$ we need to take the characteristic identities of algebras $\AlgA \in MSI(\classV)$ containing $r_b(\AlgA)$ variables. And this is always possible because defining relations of any finitely presented algebra $\AlgA$ can be written in terms of any set of generators of $\AlgA$ (cf. \cite[Corollary 7 p.223]{MaltsevBook}). Thus, by selecting in $\AlgA$ a set of generators containing exactly $r_b(\AlgA)$ elements, we can construct a characteristic identity of $\AlgA$ containing not more than $r_b(\AlgA)$ distinct variables.
\end{proof}

Let us observe that if a variety $\classV$ is finitely axiomatizable in $\classV_0$ then any optimal axiomatization of $\classV$ in $\classV_0$ contains just a finite number of axioms.

\begin{example} (comp. \cite{Skvortsov_Remark_1999}) It is well known that any finite Heyting algebra $\AlgA$ belongs to a finite slice, that is for some $n$, $\AlgA \in \classV_n$ where $\classV_n$ is defined by a characteristic identity of $(n+1)$-element chain algebra. It is also well known that varieties $\classV_n$ are locally finite. Therefore, all the varieties generated by a finite Heyting algebra admit optimal axiomatization constructed of characteristic identities.    
\end{example}

\subsubsection{Algorithmic Problems} In the cases when a subvariety $\classV$ of a locally finite variety is finitely defined, for instance, if $\classV$ is defined by a finite set of axiom or by a finite algebra, one can ask whether there is an algorithm that gives an optimal axiomatization of $\classV$. The following theorem gives a positive answer to this question.

A variety $\classV$ is called \textit{equationally decidable} if there is an algorithm that by each identity $\idn$ decides whether $\classV \models \idn$ or not. For instance, any finitely axiomatized locally finite variety is decidable.

From \cite[p.163]{Hobby_Mckenzie_Book} we also recall that the \textit{free spectrum of variety }$\classV$ is the function $\sigma_\classV$ with domain being a set of positive integers, such that $\sigma_\classV(n) = |\AlgF_\classV(n)|$ where $\AlgF_\classV(n)$ is an algebra in $\classV$ freely generated by $n$ elements. For instance, if $\classV$ is a variety of Booleana algebras then $\sigma_\classV(n) = 2^{2^n}$.   

We say that a locally finite variety $\classV$ is \textit{effectively bounded} if there is such a computable function $\beta_\classV(n)$ that $\sigma_\classV(n) \leq \beta(n)$. For instance, if $\Distr$ is a variety of distributive lattices, then $\beta_\Distr(n) = 2^{2^n}$ even though $\sigma_\Distr(n)$ is unknown.

We say that a variety $\classV$ has  \textit{decidable membership problem for finite algebras} if there is an algorithm that by given finite algebra $\AlgA$ decides whether $\AlgA \in \classV$ or not (comp. \cite{McNulty_Szekely_Willard_2008}). 

The following theorem establishes the relations between three above properties for locally finite varieties. In order to prove the theorem we also need a notion of the \textit{depth $\delta(\idt)$ of a term} $\idt$: if $\idt$ is a variable or a constant, then $\delta(\idt) = 0$ and for each $n$-ary fundamental operation $f$ 
\[
\delta(f(\idt_1,\dots,\idt_m)) = \max\lbrace \delta(\idt_i): i=1,\dots,m\rbrace +1.  
\]
In other words, $\delta(\idt)$ represents the level of nestedness of fundamental operations in $\idt$. 

\begin{theorem} Let $\classV$ be a locally finite variety. Then
\begin{itemize}
\item[(a)] If $\classV$ is equationally decidable then there is an algorithm for constructing $\classV$-free algebras of any given  finite rank;
\item[(b)] $\classV$ is equationally decidable if and only if it is effectively bounded and has a decidable membership problem for finite algebras.
\end{itemize}
\end{theorem}
\begin{proof}
(a) Since $\classV$ is equationally decidable, if we have two terms $\idt_1,\idt_2$, we can effectively answer the question whether $\classV \models \idt_1 \approx \idt_2$ or not. Our goal is using this fact to show how algorithmically to construct the free algebra $\AlgF_\classV(n)$. 
We will construct (an isomorphic copy of) $\AlgF_\classV(n)$ as a quotient algebra $\AlgT_n/\equiv$ of the algebra $\AlgT_n$ of all terms in variables $x_1,\dots,x_n$ by congruence $\equiv$, where
\[
\idt_1 \equiv \idt_2 \text{ if and only if } \classV \models \idt_1 \approx \idt_2.
\]   

We will list the terms from $\algT_n$ in the order of their increasing depth. In order to do so, first we take all the variables $x_1,\dots,x_n$. If we established that $x_i \equiv x_j$ for some $i \neq j$, then we can conclude that the variety $\classV$ is degenerate and, hence, $\AlgF_\classV(n)$ is a one-element algebra. Otherwise each variable belongs to a separate congruence class and we have $[x_1]_\equiv, \dots, [x_n]_\equiv \in \AlgT_n/\equiv$. 

Next, for each 0-ary operation (constant) $f$ we determine whether $f \equiv x_i$ for some $i$ and, therefore, $f \in [x_i]_\equiv$, or $[f]_\equiv$ is in a new congruence class. Thus, for each term of the depth 0 we determined to which congruence class this term belongs.   

If we have assigned all the terms of depth $< n$ to the $\equiv$-congruence classes we turn to the terms of the depth $n$. If all the terms of the depth $n$ belong to a congruence class containing a term of the depth $< n$, we stop. Otherwise we add new congruence classes (elements of $\AlgF_\classV(n)$) and then we take all the terms of the depth $n+1$ and repeat this process. Since the variety $\classV$ is locally finite and therefore $\AlgF_\classV(n)$ is finite (i.e. there is just a finite set of distinct $\equiv$-congruence classes), the process will stop. 

Using a simple induction on the depth of term one can prove that if for every $t \in \algT_n$   
\begin{equation}
\text{ there is a term } \idt' \text{ such that } \delta(\idt') < n \text{ and }  \idt \equiv \idt', \label{alg1}
\end{equation} 
then \eqref{alg1} is true for any term $t$ of depth $\ge n$.

By virtue of the Theorem 2 \cite[p.163]{GraetzerB}, algebra $\AlgT_n/\equiv$ is isomorphic with $\AlgF_\classV(n)$.

(b) From (a) it follows that if $\classV$ is equationally decidable, then for each $n$ we can construct the algebra $\AlgF_\classV(n)$ and, for $\AlgF_\classV(n)$ is finite, we can calculate the number of elements in it and in such a way to obtain the value of $\sigma_\classV(n)$ or $\beta_\classV(n)$ for this matter. Thus, any equationally decidable locally finite variety has a computable free spectrum and, hence, is effectively bounded.  

If a variety $\classV$ is equationally decidable, for each given finite algebra $\AlgA = \lbrace \alga_1,\dots,\alga_n\rbrace$ we can construct $\AlgF_\classV(n)$ using variables $x_\alga$ and then check whether the mapping $\phi: [x_\alga]_\equiv \mapsto \alga$ can be extended to a homomorphism of $\AlgF_\classV(n)$ onto $\AlgA$. From the properties of free algebras it easy follows that $\AlgA \in \classV$ if and only if $\phi$ can be extended to a homomorphism. Thus, if $\classV$ is equationally decidable then it is effectively bounded and has decidable membership problem for finite algebras\footnote{Note, that in general the decidability of equational theory does not yield the decidability of the membership problem for finite algebras (cf. \cite{Jezek_Decidable_1998}).}. We remind the reader that we consider only algebras with finite number of finite-ary operations, thus we can effectively check whether a mapping from one finite algebra in another is a homomorphism or not. 

Conversely, assume that $\classV$ is effectively bounded and has decidable membership problem for finite algebras. Let $\idn$ be an identity and we need to determine whether $\classV \models \idn$ or not. Let $n$ be a rank of $\idn$, that is $n$ is a number of distinct variables occurring in $\idn$. Clearly, an identity of rank $n$ is valid in $\classV$ if and only if it is valid in every $n$-generated algebra from $\classV$. Since the variety $\classV$ is effectively bounded, we know that any $n$-generated $\classV$-algebra has power $\leq \beta_\classV(n)$. Therefore, in order to determine whether an identity $\idn$ is valid in $\classV$ we can effectively list all the algebras having power not exceeding $\beta_\classV(n)$ and check whether $\idn$ hods in the algebra or not. This former is possible due to decidability of the membership problem for finite algebras. The latter is possible because all selected algebras are finite. 
\end{proof}

The following Corollary gives a simple sufficient condition for a locally finite variety to be effectively bounded and to have  decidable membership problem for finite algebras. 

\begin{cor} Let $\classV$ be a locally finite variety. If $\classV$ is finitely axiomatized then $\classV$ is effectively bounded and has decidable membership problem for finite algebras.
\end{cor}
\begin{proof} If a variety $\classV$ is finitely axiomatized then it has decidable membership problem for finite algebras: one simply needs to verify whether every axiom is valid in a given finite algebra. It is also well known that finitely axiomatized locally finite (and even finitely approximated) varieties are equationally decidable.
\end{proof}

\begin{theorem} Let $\classV_0$ be a locally finite effectively bounded variety with decidable membership problem for finite algebras. Then the following hold:
\begin{itemize}
\item[(a)] There is an algorithm that by a finite set of axioms $\lbrace \idn_i: i=1,\dots,n \rbrace$ defining a subvariety $\classV \subset \classV_0$ gives an optimal axiomatization of $\classV$ over $\classV_0$; 
\item[(b)] There is an algorithm that by a finite algebra $\AlgA$ generating a subvariety $\classV \subseteq \classV_0$ gives an optimal axiomatization of $\classV$ over $\classV_0$.
\end{itemize} 
\end{theorem}  
\begin{proof} From the proof of the Theorem \ref{optax} we know that we can obtain the optimal axiomatization by using the characteristic identities algebras from $MSI(\classV)$. So, all we need to do is to show that $MSI(\classV)$ can be constructed effectively. 

First, let us observe that under the assumption of the theorem all algebras from $MSI(\classV)$ are $k$-generated, where $k$ is a number of distinct variables occurring in the axioms $\lbrace \idn_i: i=1,\dots,n \rbrace$, or $k = |\AlgA|$. And let $m = \beta_\classV(k)$. Recall that the variety $\classV$ is effectively bounded, hence, $m$ is effectively computable.   

Next, we can take all distinct modulo isomorphism algebras from $\classV_0$ having no more than $m$ elements. And we can select only s.i. algebras and we denote this set by $\classK_0$. This can be done effectively because $\classV_0$ has decidable membership problem for finite algebras.

Now we can determine which algebras from $\classK_0$ belong to $\classV$: 
given an algebra $\AlgB \in \classK_0$ in the case (a) we can check whether 
\[
\AlgB \models \idn_i \text{ for all } i=1,\dots,n,
\]  
or, in the case (b), whether
\[
\AlgA \not\models \chi_{_{\classV_0}}(B).
\]

Let $\classK = \classK_0 \cap \classV$. If $\classK_0 = \classK$, i.e. $\classK_0 \subseteq \classV$, then $\classV = \classV_0$. Otherwise, in $\classK_0 \setminus \classK$ we can select the minimal w.r.t partial order \eqref{po} algebras. Since $\classK_0$ is a finite set, the set of algebras we have selected is finite. Using the characteristic identities of these algebras we can obtain an optimal axiomatization of the variety $\classV$.   
\end{proof}

\subsubsection{Varieties not Axiomatizable by Characteristic Identities} \label{not ax}

As we saw, the characteristic identities are $\land$-prime and it is natural to ask which varieties can be defined by such identities. The following simple proposition shows that not every variety can be defined by the characteristic identities. 

If $\classV$ is a variety, by $\classV^\circ$ we will denote a subvariety generated by finite members of $\classV$. Recall that a variety $\classV$ is called \textit{finitely approximated} (f.a.) if $\classV = \classV^\circ$. A variety $\classV$ is said to be \textit{hereditarily finitely approximated} if $\classV$ and all its subvarieties are finitely approximated. A variety $\classV$ is called \cite{Gerchiu_1972} \textit{finitely pre-approximated} if $\classV$ is not finitely approximated but all proper subvarieties of $\classV$ are finitely approximated. 

The following Lemma gives a sufficient condition for a variety $\classV$ with a TD term (or even a congruence distributive variety) to have a subvariety that cannot be defined in $\classV$ by characteristic identities (that is not an intersection of the co-splitting subvarieties).

\begin{lemma} \label{nfadef} (comp. \cite[Corollary 3 p.57]{Tomaszewski_PhD}) Let $\classV'$ be a variety and $\classV \subset \classV'$ be a not f.a. subvariety of $\classV$. Then subvariety $\classV^\circ$ is not definable in $\classV'$ by any characteristic identities of algebras from $\classV'$.
\end{lemma}
\begin{proof} Assume the contrary: the subvariety $\classV^\circ$ is defined in $\classV'$ by characteristic identities of some algebras $\AlgA_i \in \classV'; i\in I$, that is, for every algebra $\AlgB$
\begin{equation}
\AlgB \in \classV \text{ if and only if } \AlgB \models \chi_{_{\classV'}}(\AlgA_i). \label{eqaxiom}
\end{equation} 

Let us observe that $\AlgA \not\models \chi_{_{\classV'}}(\AlgA)$, hence  
\begin{equation}
\AlgA_i \in \classV' \setminus \classV. \label{axiom0}
\end{equation}

Recall that the subvariety $\classV$ is not f.a., thus there is an algebra $\AlgB \in \classV \setminus \classV^\circ$. Hence, by \eqref{eqaxiom} there is such an algebra $\AlgA_i; i \in I$ that
\begin{equation}
\AlgB \not\models \chi_{_{\classV'}}(\AlgA_i) \label{axiom1}.
\end{equation}
By virtue of the Theorem \ref{hom},
\begin{equation*}
\AlgA_i \in \CSub\CHom\AlgB,
\end{equation*}
therefore, for $\AlgB \in \classV$, we have
\begin{equation*}
\AlgA_i \in \classV,
\end{equation*}
and the latter contradicts \eqref{axiom0}.
\end{proof}

Now we can prove a rather simple criterion for a finitely approximated variety with a TD term to have all subvarieties to be definable by characteristic identities.

\begin{theorem} \label{critfa} (comp. \cite[Corollary 4 p.57]{Tomaszewski_PhD}) Let $\classV$ be a f.a. variety with a TD term. Then the following is equivalent
\begin{itemize}
\item[(a)] variety $\classV$ is hereditarily f.a.;
\item[(b)] every subvariety of $\classV$ is definable in $\classV$ by characteristic identities.
\end{itemize}
\end{theorem}
\begin{proof} 
(a) $\dimpl$ (b). Let $\classV$ be a hereditarily f.a. and $\classV' \subset \classV$ be a subvariety of $\classV$. Let us verify that the characteristic identities of all finite s.i. algebras from $\classK \bydef FSI(\classV \setminus \classV')$ define $\classV'$ in $\classV$.
Indeed, let $\classV_0$ is a subvariety defined in $\classV$ by all characteristic identities of algebras from $\classK$, i.e.
\begin{equation} 
\classV_0 = \set{\AlgA \in \classV}{\AlgA \models \chi_{_\classV}(\AlgB), \AlgB \in \classK  }. \label{v0}
\end{equation}
From the Theorem \ref{hom} it follows 
\begin{equation*}
\classV' \models \chi_{_\classV}(\AlgA) \text{ for all } \AlgA \in \classK, 
\end{equation*}  
hence,
\begin{equation*}
\classV' \subseteq \classV_0. 
\end{equation*}

We need to show that $\classV' = \classV_0$. For contradiction, assume $\classV' \subset \classV_0$. Recall that the variety $\classV$ is hereditarily f.a., that is all its subvarieties are generated by finite s.i. algebras. Hence, there is a finite s.i. algebra $\AlgB$ such that $\AlgB \in \classV_0 \setminus \classV'$. So, $\AlgB \in \classV \setminus \classV'$ and, therefore, $\AlgB \in \classK$. From the Theorem \ref{hom} it follows that $\AlgB \not\models \chi_{_\classV}(\AlgB)$, thus
\[
\AlgB \not\models \chi_{_\classV}(\AlgB) \text{, while }\AlgB \in \classV_0 \text{ and } \AlgB \in \classK 
\]
and this contradicts \eqref{v0}.

(b) $\dimpl$ (a) follows immediately from the Lemma \ref{nfadef}: if a variety $\classV$ is f.a. but not hereditary, then it contains a not f.a. subvariety and, hence, by virtue of Lemma \ref{nfadef}, it has a subvariety that cannot be defined by characteristic identities. 
\end{proof}

\begin{example} \label{heytnotax} It is well known that the variety $\Heyt$ of Heyting algebras is f.a. It is also observed by V.~Jankov \cite{Jankov_1968} that there are not f.a. subvarieties of $\Heyt$. Hence, $\Heyt$ contains the subvarieties that are not definable by characteristic identities. Moreover, in \cite{Mardaev_Embedding_1987} S.~Mardaev proved that there is a continuum of finitely pre-approximated varieties of Heyting algebras and Brouwerian lattices. Hence, there is a continuum of varieties of Heyting algebras (or Brouwerian lattices) that are not axiomatizable by characteristic formulas. Or, in terms of logics, there is a continuum of intermediate logics not axiomatizable by characteristic formulas (comp. \cite[Corollary p.128]{Tomaszewski_PhD}).
\end{example}

\subsection{Characteristic Identities vis a vis Splitting}

In the Section \ref{charsplit} we already pointed out that every characteristic formula defines a splitting. For the varieties with a TD term the situation remains the same: by the Theorem \ref{thrmord} every characteristic identity defines splitting (a co-splitting variety, to be more precise). A converse statement is not true: not every co-splitting variety is defined by a characteristic formula. Or, in other words, not every splitting algebra in a variety $\classV$ is finitely presented in $\classV$. In \cite[Theorem 7.5.16]{Kracht_Tools_1999} M.~Kracht provided a counter-example to his conjecture from \cite{Kracht_1990} that every splitting algebra is finitely presented (see also \cite{Citkin_Splitting_2012}). If we recall that every splitting identity is interderivable with a $\land$-prime identity, we can conclude that in the not finitely approximated varieties $\land$-prime identity does not have to be a characteristic identity.  

More information regarding splittings in the lattice of modal logics the reader can find in \cite{Wolter_PhD}.  For more information regarding characteristic formulas and splitting in the varieties of residuated lattices we refer the reader to \cite{Kowalski_Ono_2000,Kowalski_Ono_Remarks_2000,Kowalski_Miyazaki_All_2009,Galatos_et_Book}.

\subsection{Characteristic vis a vis Pre-True Identities}
In this section we will study the relations between characteristic and pre-true identities. We start with the definition and properties of pre-true identities.

\subsubsection{Pre-true Identities} \label{PreTrue}

It is natural to extend the definition of the pre-true formula to identities: an identity $\idn$ is \textit{pre-true in an algebra} $\AlgA$ if it is not valid in $\AlgA$ but it is valid in every proper subalgebra and every proper homomorphic image of $\AlgA$ (a \textit{subalgebra $\AlgA'$ of $\AlgA$ is proper} if $\algA' \subsetneq \algA$ and a \textit{homomorphic image of $\AlgA$ is proper} if the kernel congruence is distinct from the identity congruence). As we will see in this section even though the pre-true identities are very similar to the characteristic identities they are not the same: some pre-true identities are not characteristic and vise verse.

\begin{prop} \label{pretrueprop1} Let $\AlgA$ be an algebra and $t(x_1,\dots,x_n) \approx t'(x_1,\dots,x_n)$ be a pre-true in $\AlgA$ identity. If for some elements $\alga_1,\dots,\alga_n \in \AlgA$ 
\[
t_1(\alga_1,\dots,\alga_n) \neq t'(\alga_1,\dots,\alga_n),
\]
then 
\begin{itemize}
\item[(i)] the elements $\alga_1,\dots,\alga_n$ generate $\AlgA$; 
\item [(ii)] $t_1(\alga_1,\dots,\alga_n) \equiv t'(\alga_1,\dots,\alga_n) \quad \pmod{\mu(\AlgA)}$, where $\mu(\AlgA)$ is the monolith of $\AlgA$.
\end{itemize}
\end{prop}
\begin{proof} The proof follows immediately from the definitions of pre-true identity and monolith.
\end{proof}

If $\AlgA$ is an algebra by $\CProp(\AlgA)$ we denote the set of all proper subalgebras of $\AlgA$ and subalgebras of proper homomorphic images of $\AlgA$. Clearly, an identity $\idn$ is pre-true in $\AlgA$ if $\AlgA \not\models \idn$ and $\CProp(\AlgA) \models \idn$. Let us note that if an algebra $\AlgA$ has a pre-true identity then $\AlgA$ cannot be embedded in any algebra from $\CProp(\AlgA)$. By $\oV(\AlgA)$ we denote the variety generated by $\CProp(\AlgA)$. If $\idn$ is an identity by $\Sbs(\idn)$ we denote the set of all substitution instances of $\idn$, that is $\Sbs(\idn) \bydef \set{\sigma(\idn)}{\sigma \in \Sigma}$, where $\Sigma$ is a class of all uniform substitutions of terms for variables. 

Let us observe the following simple properties of pre-true identities that we will use in the future.

\begin{prop} \label{pretr} (comp. \cite[Corollary 4 ]{Jankov_1969})The following hold 
\begin{itemize}
\item[(a)] $\AlgA$ has a pre-true identity if and only if $\AlgA \notin \oV(\AlgA)$;
\item[(b)] If $\AlgA$ has a pre-true identity then $\AlgA$ is s.i.;
\item[(c)] If $\AlgA$ has a pre-true identity then $\AlgA$ is finitely generated.
\end{itemize}
\end{prop}
\begin{proof} 
(a) By the definition, any pre-true identity separates $\AlgA$ from $\CProp(\AlgA)$ and, hence, from $\oV(\AlgA)$.

Conversely, if $\AlgA \notin \oV(\AlgA)$ any identity that separates $\AlgA$ from $\oV(\AlgA)$ also separates $\AlgA$ from $\CProp(\AlgA)$ (for $\CProp(\AlgA) \subseteq \oV(\AlgA)$) and, hence, this identity is pre-true in $\AlgA$.

(b) By (a), $\AlgA \notin \oV(\AlgA)$. Assume for contradiction that $\AlgA$ is not s.i. Then $\AlgA$ is a subdirect product of its proper homomorphic images. Recall that every proper homomorphic image of $\AlgA$ belongs to $\CProp(\AlgA)$ and, hence, it belongs to $\oV(\AlgA)$. Thus, $\AlgA$ is a subdirect product of algebras from $\oV(\AlgA)$ and, hence, $\AlgA \in \oV(\AlgA)$ and we have arrived to contradiction.  

(c) Assume for contradiction that $\AlgA$ is not finitely generated. Let $\idn$ be an identity pre-true in $\AlgA$. Then $\AlgA \not\models \idn$. Obviously $\idn$ contains just a finite number of variables and, hence, it is refutable in a finitely generated subalgebra $\AlgA'$ of $\AlgA$. Recall that $\AlgA$ is not finitely generated. Therefore $\AlgA \not\cong  \AlgA'$. Thus, $\idn$ is invalid in a proper subalgebra $\AlgA'$ of $\AlgA$ and this contradicts that $\idn$ is a pre-true identity.
\end{proof}

Note that an algebra can have a pre-true identity and be not finitely presented. Thus, such an algebra does not have a characteristic identity at all (see Section \ref{HeytPreTrue}). On the other hand, the following holds.

\begin{prop} \label{pretrueprop2} If $\AlgA$ is an s.i. algebra finitely presented in a variety $\classV$  and $\AlgA$ is not embeddable in any member of $\CProp\AlgA$, then $\chi_{_\classV}(\AlgA)$ is pre-true in $\AlgA$. In particular, a characteristic identity of any finite s.i. algebra is pre-true in this algebra.  
\end{prop}
\begin{proof} As we know, $\AlgA \not\models \chi(\AlgA)$. On the other hand, $\CProp\AlgA \models \chi(\AlgA)$. Indeed, if $\chi(\AlgA)$ is invalid in any member of $\CProp\AlgA$, then, by the Theorem \ref{hom}, $\AlgA$ would be embeddable in a member of $\CProp\AlgA$. If $\AlgA$ is finite it cannot be embedded in any algebra from $\CProp\AlgA$, because any member of $\CProp\AlgA$ has less elements than $\AlgA$.
\end{proof}

\begin{remark} One of the essential differences between finite and infinite algebras is that the latter can be isomorphic to its own proper subalgebra or its own proper homomorphic image. In order to take this into account E.~Tomaszewski introduced \cite[Definition 3.11]{Tomaszewski_PhD} a notion of a strictly pre-true identity: an identity $\idn$ is \textit{strictly pre-true} in an algebra $\AlgA$ if $\idn$ is not valid in $\AlgA$ but it is valid in every proper subvariety of $\classV(\AlgA)$. It is easily seen that an algebra $\AlgA$ has a strictly pre-trrue identity if and only if $\AlgA$ splits the variety $\eqc(\AlgA)$ and the strictly pre-true identity defines $\oV(\AlgA)$. In the congruence distributive variety the classes of pre-true and strongly pre-true identities of finite algebras  coincide. Nevertheless, there are infinite algebras that have a strictly pre-true identity and do not have any pre-true identities simply because these algebras are own proper subalgebras.
\end{remark}

\subsubsection{Uniquely Pre-True Identities}

The goal of this section is to prove the Theorem\ref{charpretrue} that describes the relation between characteristic and pre-true identities in the hereditarily finitely approximated varieties. 

We start with an observation that the same identity may be pre-true in different algebras. Moreover, in \cite{Kuznetsov_Gerchiu, Wronski_1973} the reader can find the examples of scattered formulas, that is the formulas that are pre-true in the infinitely many algebras. This observation justifies the following definition: we say that an identity $\idn$ is \textit{uniquely pre-true in a variety} $\classV$ if $\idn$ is pre-true in exactly one algebra from $\classV$. 

\begin{theorem} \label{charpretrue} Suppose $\classV$ is a hereditarily finitely approximated variety with a TD term and $\idn$ is an identity. Then the following are equivalent:
\begin{itemize}
\item[(a)] $\idn$ is ($\classV$-equivalent with) a characteristic identity of some $\classV$-algebra ;
\item[(b)] $\idn$ is uniquely pre-true in $\classV$.
\end{itemize}  
\end{theorem}

In order to prove the theorem first we need to establish some properties of uniquely pre-true identities.

We start with the observation that the characteristic identities of finite s.i. algebras are the obvious examples of uniquely pre-true identities.

\begin{prop} \label{ucharpretrue} Any characteristic identity of a finite algebra $\AlgA$ is uniquely pre-true in $\AlgA$.
\end{prop}
\begin{proof} Let $\chi(\AlgA)$ be a characteristic identity of a finite s.i. algebra $\AlgA$. First we observe that $\chi(\AlgA)$ is pre-true in $\AlgA$. For contradiction: assume that $\chi(\AlgA)$ is refutable in some algebra $\AlgB$  that is a proper subalgebra or a proper homomorphic image of $\AlgA$. Then
\begin{equation}
|\AlgB| < |\AlgA|. \label{powers}
\end{equation}
On the other hand, by the  properties of characteristic identities, $\AlgA \in \CSub\CHom(\AlgB)$, hence
\[
|\AlgA| \leq |\AlgB|
\]
and the latter contradicts \eqref{powers}.

Next, we prove that there is no other algebra in which $\chi(\AlgA)$ is pre-true. Indeed, since $\chi(\AlgA)$ is a characteristic identity, algebra $\AlgA$ is embeddable in a homomorphic image of any algebra $\AlgB$ in which $\chi(\AlgA)$ is refutable. Hence, $\chi(A)$ cannot be pre-true in $\AlgB$, unless $\AlgB$ is isomorphic with $\AlgA$.
\end{proof}

\begin{prop}\label{fapretrue} If $\classV$ is a finitely approximated variety, $\idn$  is an identity and $\classV \not\models \idn$, then $\idn$ is pre-true in some finite algebra from $\classV$.
\end{prop}
\begin{proof} Since $\classV$ is finitely approximated it is generated by its finite members. Hence, there is a finite algebra from $\classV$ in which $\idn$ is refutable. Let $\AlgA$ be a smallest by number of elements algebra from $\classV$ in which $\idn$ is refutable. Clearly, $\idn$ is pre-true in $\AlgA$.
\end{proof}

\begin{cor} \label{fauniquely} Let $\classV$ be a finitely approximated variety and an identity $\idn$ is uniquely pre-true in $\classV$. Then $\idn$ is uniquely pre-true in some finite algebra from $\classV$.
\end{cor}

\begin{prop} \label{hfauniq} Let $\classV$ be a hereditarily finitely approximated variety with a TD term and $\idn$ be an identity. Then the following are equivalent:
\begin{itemize}
\item[(a)] $\idn$ is $\classV$-uniquely pre-true;
\item[(b)] $\idn$ is pre-true in exactly one (modulo isomorphism) finite $\classV$-algebra.
\end{itemize}
\end{prop}
\begin{proof}
(a) $\dimpl$ (b) immediately follows from the Corollary \ref{fauniquely}.

(b) $\dimpl$ (a). Let $\idn$ is pre-true in exactly one finite algebra $\AlgA$. We need to verify that there is no algebra $\AlgB \in \classV$ in which $\idn$ is pre-true. 

 For the contradiction assume that $\idn$ is pre-true in $\AlgB \in \classV$. First, observe that since $\AlgA$  has a pre-true formula $\AlgA$ is s.i. algebra. By the assumption, $\AlgA$ is finite. Let $\chi(\AlgA)$ be a characteristic identity. Next, we observe that 
\[
\AlgB \models \chi(\AlgA),
\] 
because otherwise, by the property of characteristic identity, $\AlgA \in \CSub\CHom(\AlgB)$ and the latter would contradict that $\idn$ is pre-true in $\AlgB$.

Now, let us consider $\classV' \bydef \eqc(\AlgB)$. Note that $\classV' \models \chi(\AlgA)$. On the other hand, $\AlgA \not\models \chi(\AlgA)$, hence,
\begin{equation}
\AlgA \notin \classV'. \label{anotinv}
\end{equation}

It is also easy to see that 
\[
\classV' \not\models \idn.
\] 
Recall that $\classV'$ is finitely approximated, for $\classV' \subseteq \classV$ and $\classV$ is hereditarily finitely approximated. By virtue of the Proposition \ref{fapretrue}, there is a finite algebra $\AlgB' \in \classV'$ in which $\idn$ is pre-true. In view of \eqref{anotinv}, $\AlgA \not\cong \AlgB'$.  
\end{proof}

\subsubsection{The Proof of the Theorem}

\begin{proof}
(a) $\Rightarrow$ (b). If $\classV$ is finitely approximated, then every finitely presented in $\classV$ s.i. algebra is finite. Thus, if $\idn$ is a characteristic identity of an algebra $\AlgA$, then $\AlgA$ is finite. By virtue of the Proposition \ref{ucharpretrue}, any characteristic identity $\chi(\AlgA)$ is uniquely pre-true in $\AlgA$.  

(b) $\Rightarrow$ (a). Assume $\idn$ is uniquely pre-true in $\classV$. Then, by virtue of the Corollary \ref{fauniquely}, $\idn$ is uniquely pre-true in some finite algebra $\AlgA \in \classV$. Let us prove that $\idn$ and $\chi(\AlgA)$ are $\classV$-equivalent.

Since $\AlgA \not\models \idn$, by the property of characteristic identities, we have 
\[
\idn \models \chi(\AlgA).
\]
So, we need to prove only 
\[
\chi(\AlgA) \models \idn,
\]
i.e. we need to prove that $\idn$ is valid in any algebra $\AlgB \in \classV$ in which $\chi(\AlgA)$ is valid. Assume the contrary: there is such an algebra $\AlgB$ that 
\begin{equation}
\AlgB \models \chi(\AlgA) \text{, while } \AlgB \not\models \idn. \label{notder}
\end{equation} 
Let us consider a subvariety $\classV' \subseteq \classV$ generated by $\AlgB$.  Then
\begin{equation}
\classV' \models \chi(\AlgA) \label{chartrue}
\end{equation}
and
\begin{equation}
\classV' \not\models \idn. \label{pretruenot}
\end{equation}

 Since $\classV$ is hereditarily finitely approximated, the variety $\classV'$ is finitely approximated.  From \eqref{pretruenot}, by virtue of the Proposition \ref{fapretrue}, there is an algebra $\AlgC \in \classV'$ in which $\idn$ is pre-true. Note that $\classV' \subseteq \classV$ and, therefore, $\AlgC \in \classV$. Recall that $\idn$ is $\classV$-uniquely pre-true in $\AlgA$, hence $\AlgC \cong \AlgA$. Thus, \[
\AlgA \in \classV'
\]
and, by the properties of characteristic identity,
\[
\AlgA \not\models \chi(\AlgA),
\]
that is
\[
\classV' \not\models \chi(\AlgA).
\]   
The latter obviously contradicts \eqref{chartrue}. 
\end{proof}

The Theorem \ref{charpretrue} gives us a way to construct an algorithm that recognizes whether a given identity is $\land$-prime in a hereditarily finitely approximated variety.

\begin{theorem} Let $\classV$ be a finitely axiomatized hereditarily finitely approximated variety with a TD term. Then there is and algorithm that given a list of axioms defining $\classV$ and an identity $\idn$, decides whether $\idn$ is $\land$-prime. 
\end{theorem}
\begin{proof} First, let us note that for any identity $\idn$ if $\classV \models \idn$, then, by the definition,  $\idn$ is not $\land$-prime. If $\classV \not\models \idn$, then, due to finite approximability of $\classV$, $\idn$ is $\land$-prime if and only if it id equipotent with a characteristic identity. Thus, the first step of the algorithm would br to determine whether $\classV \models \idn$. If not, then to determin whether $\idn$ is equipotent with a characteristic identity.

The sketch of the proof. In order to determine whether a given identity $\idn$ is derivable we can use a slight modification of the Harrop's algorithm: on one hand, we try to derive $\idn$ from the axioms, on the other hand, we list the finite algebras from $\classV$ in order of increasing power. If we derived $\idn$ from the axioms, it is not $\land$-prime. If $\idn$ does not follow from the axioms, the finite approximability of $\classV$ yields that we will find a finite algebra in which $\idn$ is refutable. Recall that we are listing algebras in the order of increasing power. Assume $\AlgA$ is the first algebra that  refutes $\idn$ that we found. It is easy to see that $\idn$ is pre-true in $\AlgA$. Hence, we can use the Theorem \ref{charpretrue} and try, on one hand, to derive (in the corresponding equational system) $\idn$ from $\chi(\AlgA)$, and, on the other hand, to find another finite algebra in which $\idn$ is pre-true. If $\idn$ is not uniquely pre-true, by virtue of the Proposition \ref{hfauniq}, we will find another finite algebra in which $\idn$ is pre-true.     
\end{proof}

\section{Characteristic Identities of Infinite Algebras} \label{inf}

V.~Jankov defined the characteristic formulas using diagram of a finite s.i. algebra. It is natural to ask whether one can define the formulas with the similar properties for infinite s.i. algebras. In this section we will review different possibilities. 

One of the approaches is to extend the definition of characteristic formula to finitely presented algebras. But in finitely approximated varieties every finitely presented s.i. algebra is finite. 

Another way to generalization was suggested by A.~Wronski in \cite{Wronski_Card_1974} (more recently the same approach for Johansson's algebras was used by S.~Odintsov cf., for instance, \cite[Section 7]{Odintsov_Structure_2005} or \cite[Section 6.3]{Odintsov_Lattice_2006}): instead of a formula one can use a consequence relation defined by an algebra. More precisely, each algebra $\AlgA$ defines a consequence relation in the following way: a formula $A$ (an identity $\idn$) is a \textit{consequence of a set of formulas (identities)} $\Gamma$ if every valuation that refutes $A$ (respectively $\idn$) refutes at least one formula (identity) from $\Gamma$, in written $\Gamma \models_\AlgA A$ (or $\Gamma \models_\AlgA \idn$. If we take any countable s.i. algebra $\AlgA$ and take its diagram
\[
\delta^+(\AlgA) \bydef \set{f(x_{\alga_1},\dots,x_{\alga_n}) \approx x_{f(\alga_1,\dots,\alga_n)}}{ \alga_1,\dots,\alga_n \in \algA \text{ and } f \in \Con  }
\]   
we can prove the following Proposition.

\begin{prop} \label{infimpl} (comp. \cite[Lemma 3]{Wronski_Card_1974}) Let $\AlgA$ be at most countable s.i. algebra and $\AlgB$ be an algebra from a variety $\classV$. Then the following conditions are equivalent
\begin{itemize}
\item[(a)] $\AlgA$ is embeddable in $\AlgB$;
\item[(b)] $\delta^+(\AlgA) \not\models_\AlgB x_{\algb_1} \approx x_{\algb_2}$ where $\algb_1,\algb_2$ are any two distinct elements from the monolith $\mu(\AlgA)$.
\end{itemize}
\end{prop}
\begin{proof} Assume $\phi: \AlgA \to \AlgB$ is an embedding. Then the valuation $\nu: x_\alga \mapsto \phi(\alga); \alga \in \algA$ makes all the identities from the diagram true while 
\[ 
\nu(x_{\algb_1}) = \phi(\algb_1) \neq \phi(\algb_2) = \nu(x_{\algb_2}).
\]

Conversely, suppose $\nu$ is a valuation that makes all the identities from the digram true and $\nu(x_{\algb_1}) \neq \nu(x_{\algb_2})$. Let us consider the mapping $\phi: x_\alga \mapsto \nu(x_\alga); \alga \in \algA$. The fact that all the identities from the diagram are true means that $\phi$ is a homomorphism. But, since $\phi$ sends two elements from the monolith into two distinct elements, $\phi$ is an isomorphism, that is, $\phi$ embeds $\AlgA$ in $\AlgB$ 
\end{proof}

\begin{remark} The transition from the formula to a consequence relation in logical terms means the transition from the formulas to the (structural) rules (with, perhaps, countably many premises). In algebraic terms, it signifies the transition from the varieties to implicative classes (cf. \cite{Budkin_Gorbunov_1973}). Note that if we restrict the Wronski's definition to finite algebras, we will not arrive to Jankov formulas. Instead, we will obtain the quasi-characteristic rules (for the definition see \cite{Citkin1977, Rybakov_Quasi_Characteristic_1997,Rybakov_Book}). It is worth noting that an implicative class often is much narrower then a quasivariety it generates. For instance, by the Proposition \ref{infimpl}, $\delta^+(\Z_{\infty} +\Z_2) \models_{_{\Z_\infty}} p_\omega$, hence, $\Z_{\infty} + \Z_2$ is not a member of the implicative class generated by $\Z_{\infty}$,  while $\Z_{\infty}+\Z_2 \in \qvar(\Z_\infty)$.   
\end{remark}

In \cite[Definition 5.1]{Tanaka_2007} the notion of Jankov formula is extended to complete Heyting algebras, i.e. the Heyting algebras admitting infinite joins and meets. If $\AlgA$ is a complete Heyting algebra then a subset $\classC \in 2^{\algA}$ is called basis of $\AlgA$ if 
\begin{enumerate}
\item for any $\algA' \subseteq \algA$ there exists $\algC \in \classC$ such that $\lor \algA' = \lor \algC$ and for any $\algc \in \algC$ there exists $\alga \in \algA'$ such that $\alga > \algc$;
\item for any $\algA' \subset \algA$ there exists $\algC \in \classC$ such that $\land \algA' = \land \algC$ and for any $\algc \in \algC$ there exists $\alga \in \algA'$ such that $\alga < \algc$. 
\end{enumerate}
And the characteristic formula for a complete s.i. Heyting algebra $\AlgA$ and its basis $\classC$ is defined as
\[
\begin{split}
\chi(\AlgA,\classC) \bydef  & \bigwedge_{\algC \in \classC}(p_{\lor \algC} \eqv \bigvee_{\algc \in \algC} p_\algc ) \land \bigwedge_{\algC \in \classC}(p_{\land \algC} \eqv \bigwedge_{\algc \in \algC} p_\algc ) \\
&  \land \bigwedge_{\alga,\algb \in \algA} (p_{\alga \impl \algb} \eqv (p_\alga \impl p_\algb)) \land \bigwedge_{\alga \in \algA}(p_{\neg \alga} \impl \neg p_\alga). 
\end{split}
\]
Using the argument similar to the one that was used in the proof of Jankov Theorem one can prove (cf. \cite[Proposition 5.1]{Tanaka_2007}) that an s.i. complete Heyting algebra $\AlgA$ is embeddable into a homomorphic image of a complete Heyting algebra $\AlgB$ if and only if $\AlgB \not\models \chi(\AlgA,\classC)$ and the homomorphism and the embedding are continuous. 

\subsection{Locally Characteristic Identities}

As we saw in the previous section, the attempts to construct a characteristic formula for an infinite algebra requires either to use infinite formulas or to replace a formula with a consequence relation. If a consequence relation is finitary, that is $\Gamma \vdash A$ if and only if $\Gamma' \vdash A$ for a finite subset $\Gamma' \subseteq \Gamma$, we can attempt to use a set of formulas $\set{\land \Gamma' \impl A}{\Gamma' \subseteq \Gamma, \Gamma' \text{ is finite}}$ (c.f., for instance, \cite{Skura1992}).   An alternative approach to characteristic formulas of infinite algebras is by using diagrams of partial subalgebras (comp. \cite{Tomaszewski_PhD,Citkin_Char_2012}). Let us consider the latter approach. 

If $\AlgA = \lbr \algA; \Con \rbr$ is an algebra then any subset $\algA' \subset \algA$ can be turned into a partial algebra $\AlgA' = \lbr \algA'; \Con \rbr$ simply by restricting fundamental operations of $\AlgA$ to $\algA'$. In this case $\AlgA'$ is called \textit{partial} (e.g. \cite{Blok_Alten_Finite_2002}) or \textit{relative} (e.g. \cite{GraetzerB}) \textit{subalgebra} of $\AlgA$. By $\CRel(\AlgA)$ we denote a set of all partial subalgebras of $\AlgA$. Thus $\CFin\CRel(\AlgA)$ is a set of all finite non-degenerate subalgebras of $\AlgA$.

With each partial algebra $\AlgA'$ we can associate a positive diagram
\[
\begin{split}
\delta^+(\AlgA') = & \lbrace f(x_{\alga_1},\dots,x_{\alga_n}) \approx	 x_{f(\alga_1,\dots,\alga_n)}~|~ \alga_1,\dots,\alga_n \in \algA', f \in \Con \\
& \text{ and } f(\alga_1,\dots,\alga_n) \text{ is defined in } \AlgA' \rbrace.
\end{split}
\]

Recall that a mapping $\phi: \AlgA \impl \AlgB$ of partial algebra $\AlgA$ in partial algebra $\AlgB$ is a homomorphism if $\phi$ preserves all fundamental operations whenever the operation is defined in $\AlgA$. 

\begin{prop} \label{parthom} If $\AlgA$ and $\AlgB$ are partial algebras and there is a valuation $\nu$ in $\AlgB$ such that $\AlgB \models \nu(\delta^+(\AlgA))$, then there is a homomorphism $\phi: \AlgA \impl \AlgB$.  
\end{prop}
\begin{proof} It is not hard to see that the mapping $\phi: \alga \mapsto \nu(x_\alga)$ is, indeed, a homomorphism.
\end{proof}

Let $\classV$ be a variety with a  TD term $td(x,y,z)$. A \textit{partial algebra $\AlgA$ belongs to} $\classV$ (in written $\AlgA \in \classV$) if there is 1-1-homomorphism of $\AlgA$ in some algebra from $\classV$.  Note that for partial algebras 1-1-homomorphism does not have to be an embedding.

With each finite non-degenerate partial algebra $\AlgA$ and each pair of distinct elements $\algb,\algc \in \algA, \algb \neq \algc$ we associate a \textit{locally characteristic identity} in variables $x_\alga; \alga \in \algA$: let 
\[
\delta^+(\AlgA) = \set{ t_i \approx t'_i}{ 1 \leq i \leq m};
\] 
be a diagram $\ut \bydef t_1,\dots,t_m$ and $\ut' \bydef t'_1,\dots,t'_m$ then 
\[
\chi(\AlgA, x_\alga,x_\algb) \bydef td(\ut,\ut',x_\algb) \approx td(\ut,\ut',x_\algc).
\]
In other words, we construct a characteristic identity in the same way as we did using defining relations for finitely presented algebras. However, now we are treating the diagram identities as defining relations. Let us note the following.


\begin{prop}\label{chrefut} Let $\AlgA$ be a non-degenerate finite partial algebra, $\algb,\algc \in \algA$ and $\algb \neq \algc$. Then
\[
\AlgA \not\models \chi(\AlgA,x_\algb,x_\algc).
\]
\end{prop}
\begin{proof} It is not hard to see that the valuation $\nu: x_\alga \mapsto \alga; \alga \in \algA$ refutes $\chi(\AlgA,x_\algb,x_\algc)$. Indeed, on one hand we have the list of $\nu(f(x_{\alga_1},\dots,x_{\alga_n}))$ for all fundamental operations and all sets of elements where these operations are defined. And, clearly, 
\[
\nu(f(x_{\alga_1},\dots,x_{\alga_n})) = f(\nu(x_{\alga_1},\dots,\nu(x_{\alga_n}))) = f(\alga_1,\dots,\alga_n).
\]
On the other hand, we have 
\[
\nu(x_{f(\alga_1,\dots,\alga_n)}) = f(\alga_1,\dots,\alga_n).   
\]
Hence, under the valuation $\nu$ the identity $\chi(\AlgA,x_\algb,x_\algc)$ is of form
\[
td(\ua,\ua,\algb) \approx td(\ua,\ua,\algc)
\]
and, for $\algb \neq \algc$, by \eqref{alg1}
\[
td(\ua,\ua,\algb) \neq td(\ua,\ua,\algc).
\]
\end{proof}

For each (at most countable) algebra $\AlgA$ we define a \textit{characteristic set} by letting
\[
CS(\AlgA) \bydef \set{ \chi(\AlgA',x_\algb,x_\algc)}{ \AlgA' \in \CFin\CRel(\AlgA), \algb,\algc \in \AlgA', \algb \neq \algc }.
\]

The following Proposition is a generalization of the Proposition \ref{prophom}.

\begin{prop} \label{genhom} Let $\AlgA$ be a non-degenerate finite partial algebra, $\AlgB$ be a partial algebra, $\alga', \alga'' \in \AlgA$ and $\alga' \neq \alga''$. Then  
\[
\begin{split}
& \AlgB \not\models \chi(\AlgA,x_{\alga'},x_{\alga''}) \text{ if and only if there is a homomorphism } \\
& \phi: \AlgA \to \AlgB' \text{ such that } \AlgB' \in \CHom(\AlgB) \text{ and }\phi(\alga') \neq \phi(\alga'').
\end{split}
\] 
\end{prop}
\begin{proof} Immediately from the Proposition \ref{chrefut} and the properties of homomorphism it follows that if $\phi: \AlgA \to \AlgB$ is a homomorphism and $\phi(\alga) \neq \phi(\alga')$, then the valuation $\nu: x_\alga \mapsto \phi(\alga); \alga \in \AlgA$ is a refuting valuation for $\chi(\AlgA,x_\algb,x_\algc)$.

Let $\AlgB \not\models \chi(\AlgA,x_\algb,x_\algc)$. We need to show that there is such a homomorphism $\ophi: \AlgA \to \AlgB' \in \CHom(\AlgB)$ that $\ophi(\algb) \neq \phi(\algc)$.

Indeed, since $\AlgB \not\models \chi(\AlgA,x_\algb,x_\algc)$, there is a valuation $\nu: x_\alga \mapsto \algb_\alga \in \AlgB$ refuting the identity $\chi(\AlgA,x_\algb,x_\algc)$. That is, 
\begin{equation}
 td(\ut(\nu(\ux)),\ut'(\nu(\ux)),\algb) \neq td(\ut(\nu(\ux)),\ut'(\nu(\ux)),\algb'), \label{hom1}
\end{equation}
where $\nu(\ux)$ denotes $\nu(x_1),\dots,\nu(x_n)$.

Let us consider the mapping $\phi: \alga \mapsto \nu(x_\alga)=\algb_\alga$:

\[
\ctdiagram{
\ctv 0,60:{x_{\alga}}
\ctv 70,60:{\alga}
\ctv 70,0:{\algb_\alga}
\ctel 0,60,70,0:{\nu}
\ctet 68,60,70,60:{}
\ctet 2,60,0,60:{}
\cter 70,2,70,0:
\ctnohead\ctdot
\def\zzctdrawdotedge{\drawdotedge{2.5pt}0}
\cten 70,60,70,0:
\ctv 75,32:{\phi}
\ctnohead\ctdash
\ctet 70,60,0,60:{}
\ctet 0,60,70,60:{}
}
\] 
\begin{center} Fig 1. The homomorphism  \end{center}
Our goal is to demonstrate that 
\begin{itemize}
\item[(a)] $\phi$ can be extended to a natural homomorphism $\ophi: \AlgA \to \AlgB/\theta$ of a suitable congruence $\theta$ ;
\item[(b)] $\algb' \neq \algb''$, where  $\algb' = \ophi(\alga')$ and $\algb'' = \ophi(\alga'')$.
\end{itemize}

\textbf{Proof of (a).} From \eqref{hom1} and the definitions of $\nu$ and $\phi$ it follows that for all fundamental operations and all the sets of elements where these operations are defined we have 
\[
\begin{split}
 \nu(f(x_{\alga_1},\dots,x_{\alga_n})) = & f(\nu(x_{\alga_1}),\dots,\nu(x_{\alga_n})) = \\
& f(\phi(\alga_1),\dots,\phi(\alga_n)) = f(\algb_{\alga_1},\dots,\algb_{\alga_n}).
\end{split}
\]
On the other hand,
\[
\nu(x_{f(\alga_1,\dots,\alga_n)}) = \phi(f(\alga_1,\dots,\alga_n)) = \algb_{f(\alga_1,\dots,\alga_n)}.
\]

Let $\theta$ be a congruence generated by pairs $(f(\algb_{\alga_1},\dots , \algb_{\alga_n}), \algb_{f(\alga_1,\dots,\alga_n)})$. Now we can apply \eqref{adpc} and conclude that in $\AlgB$ 
\[
f(\algb_{\alga_1},\dots,\algb_{\alga_m} \equiv \algb_{f(\alga_1,\dots,\alga_m)}  \pmod{\theta} \text{, while } \algb' \not\equiv \algb'' \pmod{\theta}
\]
for all fundamental operations and all lists of elements $\alga_1,\dots,\alga_m \in \algA'$ for which $f(\alga_1,\dots,\alga_m)$ is defined in $\AlgA'$. 

Let us consider the quotient algebra $\AlgB/\theta$ and the mapping 
\[
\ophi:  \alga \mapsto [\algb_\alga]_\theta.
\] 
Note that $\ophi$ is a homomorphism, for every fundamental operation $f$ and any elements $\alga_1,\dots,\alga_n \in \AlgA$ such that $f(\alga_1,\dots,\alga_n)$ we have
\[
\begin{split}
& \ophi(f(\alga_1,\dots,\alga_n)) = [\phi(f(\alga_1,\dots,\alga_n))]_\theta = [\algb_{f(\alga_1,\dots,\alga_n)}]_\theta = \\
& [f(\algb_{\alga_1},\dots,\algb_{\alga_n})]_\theta = f([\algb_{\alga_1}]_\theta,\dots,[\algb_{\alga_n}]_\theta) = f(\ophi(\alga_1),\dots,\ophi(\alga_n)).
\end{split}
\]

\textbf{Proof of (b)} From \eqref{hom1} and the definition of TD term it follows that elements $\algb'$ and $\algb''$ do not belong to the same $\theta$-congruence class, hence,
\[
\ophi(\alga')  = [\algb']_\theta \neq [\algb'']_\theta = \ophi(\alga''),
\]
so $\ophi$ is such a homomorphism of $\AlgA$ in $\AlgB/\theta$ that
$[\algb']_\theta \neq [\algb'']_\theta$.
\end{proof}

\begin{cor} \label{1-1-hom}Let $\AlgA$ be a non-degenerate finite partial algebra and $\classV$ be a variety. If 
\[
\classV \not\models \chi(\AlgA,x_{\alga'},x_{\alga''})
\]
for every distinct $\alga',\alga'' \in \algA$, then there is 1-1-homomorphism of partial algebra $\AlgA$ in an algebra $\AlgB \in \classV$.
\end{cor}
\begin{proof} Let $\alga',\alga''$ are distinct elements from $\AlgA$ and $\AlgB(\alga',\alga'')$ be an algebra in which $\chi(\AlgA,x_{\alga'},x_{\alga''})$ is refuted. Then, by virtue of the Proposition \ref{genhom}, there is such a homomorphism
\[
\phi_{\alga',\alga''}: \AlgA \to \AlgB(\alga',\alga'')
\]
that $\phi_{\alga',\alga''}(\alga') \neq \phi_{\alga',\alga''}(\alga'')$. It is not hard to see that there is a 1-1homomorphism of $\AlgA$ in the direct product of algebras $\AlgB(\alga',\alga'')$. 
\end{proof}

The most important property of the characteristic set is the following.

\begin{theorem} \label{thembed} Let $\classV'$ be a variety with a TD term, $\classV \subseteq \classV'$ be a subvariety and $\AlgA \in \classV'$ be an algebra. Then the following are equivalent
\begin{itemize}
\item[(a)] $\AlgA \in \classV$;
\item[(b)] each identity $\idn \in CS(\AlgA)$ is refuted in $\classV$.
\end{itemize} 
\end{theorem}

\begin{proof} (a) $\dimpl$ (b) follows immediately from the Proposition \ref{chrefut}. 

(b) $\dimpl$ (a). Assume the contrary: there is such an algebra $\AlgA \notin \classV$ that each identity $\idn \in CS(\AlgA)$ is refuted in $\classV$. Then there is an identity $r(\ux) \approx r'(\ux)$ that separates $\AlgA$ from $\classV$, that is such an identity $r \approx r'$ that 
\begin{equation}
\classV \models r \approx r'  \text{ while } \AlgA \not\models r \approx r'. \label{contr}
\end{equation}
Let $\ua$ be a list of elements of $\AlgA$ on which the identity $r \approx r'$ fails, i.e. 
\begin{equation}
r(\ua) \neq r'(\ua). 
\end{equation}
Now we can take a finite subset $\algA' \subseteq \algA$ of all the elements from $\algA$ necessary to compute values of $r(\ua)$ and $r'(\ua)$. And we can regard $\algA'$ as a partial subalgebra of $\AlgA$. Clearly, $\AlgA'$ is finite and non-degenerate. Let us consider the characteristic identity
\[
\chi(\AlgA',x_{t(\ua)},x_{t'(\ua)}).
\] 
By (b),
\[
\classV \not\models \chi(\AlgA',x_{r(\ua)},x_{r'(\ua)}).
\] 
Hence, for some algebra $\AlgB \in \classV$
\[
\AlgB \not\models \chi(\AlgA',x_{r(\ua)},x_{r'(\ua)}).
\]
By virtue of the Proposition \ref{genhom}, there is a homomorphism $\phi: \AlgA' \to \AlgB'$ such that $\AlgB' \in \CHom(\AlgB)$ and $\phi(r(\ua)) \neq \phi(r'(\ua))$. Therefore, $r(\phi(\ua)) \neq r'(\phi(\ua))$. Let $\ub \bydef \phi(\ua)$. Then $r(\ub) \neq r'(\ub)$, that is
\[
\AlgB' \not\models r \approx r'.
\]
Recall that $\AlgB' \in \CHom(\AlgB)$ and, for $\AlgB \in \classV$, we have $\AlgB' \in \classV$. Thus,
\[
\classV \not\models r \approx r' 
\]
and the latter contradicts \eqref{contr}.
\end{proof}

We can rephrase the above theorem in the way that will be more convenient for applications.

\begin{cor}\label{embed1} Let $\classV'$ be a variety with a TD term, $\classV \subseteq \classV'$ be a subvariety and $\AlgA \in \classV'$ be an algebra. Then the following are equivalent
\begin{itemize}
\item[(a)] $\AlgA \in \classV' \setminus \classV$;
\item[(b)] there is an identity $\idn \in CS(\AlgA)$ such that $\classV \models \idn$.
\end{itemize} 
\end{cor}

\begin{cor} \label{partref} Let $\classV$ be a variety with a TD term, $\AlgA \in \classV$ be an algebra and $\idn$ be an identity. Then the following are equivalent
 \begin{itemize}
\item[(a)] $\AlgA \not\models \idn$;
\item[(b)] $\idn \models_\classV \chi(\AlgA',x_\alga,x_\algb)$ for some finite partial subalgebra $\AlgA'$ and some distinct elements $\alga,\algb \in \algA$; 
\end{itemize}
\end{cor}

\begin{proof} (b) $\Rightarrow$ (a). By the Proposition \ref{chrefut}, $\AlgA' \not\models \chi(\AlgA',x_\alga,x_\algb)$ and, since $\AlgA'$ is a partial subalgebra of $\AlgA$, we have $\AlgA \not\models \chi(\AlgA',x_\alga,x_\algb)$. Hence, if (b) holds $\idn$ cannot be valid in $\AlgA$.

(a) $\Rightarrow$ (b). Assume $\idn \bydef t \approx t'$ and $\AlgA \not\models t \approx t'$. Then there is a valuation $\nu$ in $\AlgA$ refuting $\idn$, that is $\nu(t) \neq \nu(t')$. Let $\AlgA'$ be a finite partial subalgebra of $\AlgA$ consisting of all the elements of $\AlgA$ needed to refute $\idn$ by valuation $\nu$. And let $\algb = \nu(t)$ and $\algb' = \nu(t')$. Clearly $\algb,\algb' \in \algA'$ and $\AlgA'$ is non-degenerate, for $\algb \neq \algb'$. Let us consider $\chi(\AlgA',x_\algb,x_{\algb'})$.

If in any algebra $\AlgB \in \classV$ the identity $\chi(\AlgA',x_\algb,x_{\algb'})$ is refuted, i.e.
\[
\AlgB \not\models \chi(\AlgA',x_\algb,x_{\algb'}),
\]  
then, by virtue of the Proposition \ref{parthom}, there is a homomorphism $\phi: \AlgA' \to \AlgB'$ where $\AlgB'$ is a homomorphic image of $\AlgB$ and $\phi(\algb) \neq \phi(\algb')$. Thus,
\[
\AlgB' \not\models t \approx t'.
\]
and this observation concludes the proof. \end{proof}

\subsection{Axiomatization by Locally Characteristic Identities}

As we already know from the Section \ref{not ax}, not every subvariety of a variety $\classV$ with a TD term can be axiomatized by characteristic identities. But, using the locally characteristic identities, one can axiomatize any subvariety of $\classV$. In this respect the locally characteristic identities are similar to Zakharyaschev's canonical formulas (for definitions see \cite{Chagrov_Zakh}; in the Section \ref{canon} we will discuss the relations between locally characteristic identities and canonical formulas).  

\begin{theorem} \label{fomdecomp} (comp. \cite[Theorem 9.43]{Chagrov_Zakh}) Let $\classV$ be a variety with a TD term and $\idn$ be such an identity that $\classV \not\models \idn$. Then there is a finite set of locally characteristic identities $\Gamma$ such that $\Gamma \sim_\classV \idn$. 
\end{theorem}

\begin{proof} Let $\idn \bydef t \approx t'$ be an identity such that $\classV \not\models \idn$. We  assume that $\idn$ contains variables only from $X = \set{ x_i}{ 1 \leq i \leq m }$. Let $V$ be a class of all valuations $\nu: X \to \AlgA \in \classV$ refuting $\idn$, that is, $\nu(t) \neq \nu(t')$. If $\idn$ contains $k$ subterms, then $\nu$ refutes $\idn$ in a partial subalgebra of $\AlgA$ containing not more than $k$ elements. Let $\classK$ be a set of all distinct modulo isomorphism such partial subalgebras. Clearly, $\classK$ is a finite set. Suppose $V$ is a class of all refuting valuations in $\classK$. Let
\[
\Gamma \bydef \set{ \chi(\AlgA',x_{\nu(t)},x_{\nu(t')})}{ \nu \in V, \AlgA' \in \classK },
\]
i.e. $\Gamma$ is a class of all locally characteristic identities representing all possible refutations of $\idn$ in $\classV$. Let us demonstrate that
\[
\Gamma \sim_\classV \idn.
\]

In order to demonstrate that $\Gamma \models_\classV \idn$, we need to verify that for any algebra $\AlgA \in \classV$, $\AlgA \not\models \idn$ yields $\AlgA \not\models \chi(\AlgA',x_\alga,x_\algb)$ for some $\chi(\AlgA',x_\alga,x_\algb) \in \Gamma$. Indeed, if $\AlgA \not\models \idn$, then $\idn$ is refutable in some not more than $k$-element partial subalgebra $\AlgA'$ of $\AlgA$. Isomorphic copy of $\AlgA'$ belongs to $\classK$. Hence, there is a valuation $\nu$ refuting $\idn$ in $\classK$. Let us consider the corresponding locally characteristic identity from $\Gamma$. By virtue of the Proposition \ref{chrefut}, this identity is invalid in $\AlgA'$ and, hence, it is invalid in $\AlgA$ too.
 
Next, we prove that $\idn \vDash_\classV \chi(\AlgA',x_{\nu(t)},x_{\nu(t')})$ for every $\nu \in V$. Indeed, if $\AlgB \not\models \chi(\AlgA',x_{\nu(t)},x_{\nu(t')})$, then, according to the Proposition \ref{genhom}, there is a homomorphism $\phi: \AlgA' \to \AlgB'$ where $\AlgB' \in \CHom(\AlgB)$ and 
\[
\phi(\nu(t)) \neq \phi(\nu(t')).
\]
Thus, $\nu \circ \phi$ refutes $\idn$ in $\AlgB'$. Clearly, $\idn$ cannot be valid in $\AlgB$.
\end{proof}

As a simple consequence of the above theorem we obtain the following:

\begin{theorem} \label{locaxiom} Let $\classV$ be a variety with a TD term. Then any subvariety of $\classV$ can be axiomatized over $\classV$ by locally characteristic identities. Moreover, if a subvariety $\classV' \subset \classV$ is finitely axiomatizable over $\classV$ it is finitely axiomatizable by locally characteristic identities.
\end{theorem}

\begin{remark}
The analysis of the proof of the Theorem \ref{fomdecomp} shows that if there is an algorithm for listing of all finite non-degenerate algebras of $\classV$, then there is an algorithm that by an identity $\idn$ such that $\classV \not\models \idn$ gives
the equivalent set of locally characteristic identities.
\end{remark} 

On the set of non-degenerate finite partial algebras one can introduce a quasi-order in the way similar to the one we used for finite s.i. algebras: if $\AlgA,\AlgB$ are finite non-degenerate partial algebras then 
\begin{equation}
\AlgA \leq \AlgB \bydef \AlgB \not\models \chi(\AlgA,x_\alga,x_\algb) \text{ for all } \alga,\algb \in \AlgA, \ \alga \neq \algb. \label{qusiord}
\end{equation}

\begin{prop} The relation $\leq$ defined by \eqref{qusiord} is a quasi-order.
\end{prop}
\begin{proof}
The reflexivity follows straight from the Proposition \ref{chrefut}. 

Let $\AlgA, \AlgB, \AlgC$ are finite non-degenerate partial algebras and $\AlgA \leq \AlgB$ and $\AlgB \leq \AlgC$. Let also $\alga,\algb \in \AlgA$ and $\alga \neq \algb$. Then by the definition of $\leq$ we have $\AlgB \not\models \chi(\AlgA,x_\alga,x_\algb)$.
\end{proof}

It is easy to see that the introduced above quasi-order satisfies the descending chain condition, thus in any set of partial algebras we can select a subset of all minimal elements. More precisely, if $\classV' \subset \classV$ is a subvariety, we can take a set of all minimal partial subalgebras of algebras from $\classV \setminus \classV'$ locally characteristic identities of which are valid in $\classV$. These identities give an axiomatization of $\classV'$ over $\classV$. It is worth noticing that, in  contrast to the characteristic identities, the locally characteristic identities of the minimal partial algebras may not be independent. Indeed, by virtue of the Theorem \ref{locaxiom}, any variety with a TD term (or a corresponding logic) would be independently axiomatizable, which is not true: in \cite{Chagrov_Zakh_1995_indep} it is observed that not every intermediate or normal modal logic is independently axiomatizable.

\subsection{Characteristic vis a vis Canonical Formulas} \label{canon}

We say that a set $\classI$ of identities is \textit{$a$-complete}\footnote{in \cite{Tomaszewski_PhD} such sets are called "sufficiently rich".} over a given variety $\classV$ if any subvariety of $\classV$ can be axiomatized by a subset of identities from $\classI$. A trivial example of an $a$-complete set is a set of all identities. As it follows from the Theorem \ref{locaxiom}, a set of all locally-characteristic identities is $a$-complete over any variety with a TD term. On the other hand, the set of characteristic identities is $a$-complete not over every variety (see e.g. the Theorem \ref{critfa}). For instance, as we saw in  the Example \ref{heytnotax}, the set of characteristic identities is not $a$-complete over variety $\Heyt$ of Heyting algebras. While the set of identities corresponding to Zkharyaschev's canonical formulas (see e.g. \cite{Chagrov_Zakh}) is indeed $a$-complete. In this section we first give a simple sufficient condition for a set of locally characteristic identities to be $a$-complete and then we will discuss the relations between locally-characteristic an canonical identities. 

\subsubsection{A Sufficient Condition for $a$-Completion}

Let $\classV$ be a variety with a TD term. Given partial algebra $\AlgA, \AlgB$ we write $\AlgA \unlhd \AlgB$ if there is 1-1-homomorphism $\AlgA \impl \AlgB$. We say that a class $\classK$ of finite partial algebras is $p$-\textit{complete in} $\classV$ (we will often omit a reference to $\classV$), if for each algebra $\AlgA \in \classV$ and each finite partial algebra $\AlgB$
\begin{equation}
\text{if } \AlgB \unlhd \AlgA \text{ then there is } \AlgC \in \classK \text{ such that } \AlgB \unlhd \AlgC \unlhd \AlgA. \label{partproj}
\end{equation}

\begin{theorem} Let $\classV$ be a variety with a TD term and $\classK$ a set $p$-complete in $\classV$. Then The set $\classL$ of all characteristic identities from $\classK$ is $a$-complete. 
\end{theorem}
\begin{proof} If $\classV' \subset \classV$ is a subvariety, then, for each algebra $\AlgA \in \classV \setminus \classV'$ there is an identity $\idn$ such that $\classV' \models \idn$, while $\AlgA \not\models \idn$. So, there is a partial subalgebra $\AlgB \unlhd \AlgA$ in which $\idn$ is refutable. By \eqref{locaxiom}, there is a partial algebra $\AlgC$ such that $\AlgB \unlhd \AlgC \unlhd \AlgA$. Clearly, $\AlgC \not\models \idn$. Hence, one of the characteristic identities (let say, $\chi(\AlgB)$), of $\AlgC$ is valid in $\classV$, for $\classV \models \idn$. On the other hand, any characteristic identity of $\AlgC$ is refutable in $\AlgC$. Thus, $\chi(\AlgB)$ separates $\AlgA$ from $\classV'$ and this observation completes the proof.   
\end{proof}

Sometimes, the existence of convenient $p$-complete classes is related to the following properties of subreducts.

Let $\Con' \subseteq \Con$ be a subset of operations (connectives) and $\classV$ be a variety of algebras in  the signature $\Con$. We say that $\classV$ is $\Con'$\textit{-locally finite} if any finitely generated $\Con'$-subreduct of any algebra from $\classV$ is finite.

\begin{example} Since every finitely generated distributive lattice is finite (see, for instance, \cite{Birkhoff}), the variety $\Heyt$ of Heyting algebras is $\lbrace \land, \lor \rbrace$-locally finite. Since every finitely generated Brouwerian semilattice is finite (see, \cite{McKay_Finite_1967}), the variety $\Heyt$ is $\lbrace \land, \impl \rbrace$-locally finite. It is not hard to see that the latter entails that the variety $\Heyt$ is also $\lbrace \land, \impl, \neg \rbrace$-locally finite and $\lbrace \land, \impl, \bot \rbrace$-locally finite. 
\end{example} 

Let us observe, that if a variety $\classV$ is $\Con'$-locally finite, then the class of all finite partial algebras in which the operation from $\Con'$ are totally defined is $p$-complete: every finite partial subalgebra $\AlgB$ of an algebra $\AlgA$ can be extended as a $\Con'$-reduct to a partial subalgebra $\AlgC \unlhd \AlgA$ in which all operations from $\Con'$ are totally defined. Due to $\Con'$-local finiteness, $\AlgC$ is finite and, clearly, $\AlgB \unlhd \AlgC$. So, the following holds.

\begin{theorem} Let $\classV$ is a $\Con'$-locally finite variety with a TD term. And let $\classK$ be a class of all partial algebras from $\classV$ in which operations from $\Con'$ are totally defined. Then $\classK$ is $p$-complete and any proper subvariety of $\classV$ can be axiomatize by characteristic identities of algebras from $\classK$.  
\end{theorem}

For example, since the variety $\Heyt$ of Heyting algebras is $\land,\impl,\neg$-locally finite, every proper variety of Heyting algebras can be axiomatized by the characteristic formulas of partial Heyting algebras in which operations $\land,\impl,\neg$ are totally defined. Such characteristic formulas are (interderivable with) Zakharyaschev's canonical formulas (comp. \cite{Tomaszewski_PhD,Bezhanishvili_G_N_2009}). For more information regarding relations between characteristic formulas of partial Heyting algebras and Zakharyaschev's canonical formulas see \cite{Citkin_Char_2012}.

\section{Characteristic Formulas as a Mean of Refutation}

It was first observed by G.~Birkhoff \cite{Birkhoff_1935} that the relation $\Vdash$ can be defined syntactically. If we take the following inference rules
\begin{itemize}
\item[(a)] $\Vdash t \approx t$
\item[(b)] $ t \approx r \Vdash r \approx t$ 
\item[(c)] $ t \approx r, r \approx s \Vdash t \approx s$
\item[(d)] $t_i \approx r_i, i=1,\dots n \Vdash f(t_1,\dots,t_n) \approx f(r_1,\dots,r_n)$.
\end{itemize}
In a usual way we can introduce the notion of inference and a derivability relation $\Vdash$. The Birkhoff's completeness theorem states that if a set of identities $\classI$ defines a variety $\classV$, then for each set of identities $\Gamma$ and an identity $\idn$ 
\[
\classI \cup \Gamma \Vdash \idn \text{ if and only if } \Gamma \models_\classV \idn.
\]
Thus, any set of identities $\classI$ defines an equational logic $Eq(\classI) = \set{\idn}{\classI \Vdash \idn}$. If $\classV$ is a variety, by $Eq(\classV)$ we will denote the equational logic of $\classV$, that is, the set of all identities valid in $\classV$. If $\classI$ is a finite set (the equational logic is finitely axiomatizable), the question arises whether $Eq(\classI)$ is decidable, that is, whether there is an algorithm recognizing for each given identity $\idn$ whether $\Vdash \idn$ holds. If a variety $\eqc(\classI)$ defined by the identities $\classI$ is finitely approximated, then the logic $Eq(\classI)$ is decidable. In \cite{Harrop_1958} R.~Harrop suggested the following algorithm: given an identity $\idn$, run in parallel two processes
\begin{itemize}
\item[(der)] enumerate all derivations from $\classI$ and check whether one of them ends with $\idn$ (that is $\classI \Vdash \idn$)
\item[(ref)] enumerate all finite algebras from $\AlgA \in \eqc\classI)$ (which is possible due to finite axiomatizability) and check whether $\AlgA \not\models \idn$ (that is $\idn$ is refutable in $\eqc(\classI)$ and, hence, $\classI \nVdash \idn$).
\end{itemize}
Due to finite approximability, one of the processes will always halt.

Now, let us assume that the variety $\eqc(\classI)$ has a TD term. Then, instead of attempting to refute the identity $\idn$ in an algebra $\AlgA$, we can try to derive $\chi(\AlgA)$ from $\classI \cup \{\idn\}$. Thus, the second process of Harrop's algorithm can be modified in the following way: 
\begin{itemize}
\item[(ref')] Enumerate all derivations from $\classI \cup \{\idn\}$ and check whether it ends with a characteristic identity of a finite s.i. algebra from $\eqc(\classI)$ (provided there is an algorithm to recognize whether a given identity is a characteristic identity of an algebra from $\eqc(\classI)$).
\end{itemize}
This gives us an idea how the characteristic identities can be used as a syntactic mean of refutation. In this section we will discuss this in more details.  

\subsection{r-Complete Sets} \label{rcomplete}

Let $\classV$ be a variety and $\classI$ be a set of identities. We say that $\classI$ \textit{is r-complete w.r.t.} $\classV$ if for any identity $\idn$ such that $\classV \not\models \idn$ there is an identity $\idr \in \classI$ such that $\idn \models_\classV \idr$ (often we will omit the reference to $\classV$).

\begin{example} For any variety $\classV$ the set $\set{\idn}{\classV \not\models \idn}$ is a trivial r-complete set. If a variety $\classV$ is finitely approximated and has a TD term then $\set{\chi(\AlgA)}{\AlgA \in FSI(\classV)}$ is an r-complete set.
\end{example}

Let us note the following simple property of r-complete sets.

\begin{prop} \label{finrcompl} If a variety $\classV$ has a finite r-complete set then every r-complete w.r.t. $\classV$ set contains a finite r-complete subset. 
\end{prop}
\begin{proof} Suppose $\idn_1,\dots,\idn_n$ is an r-complete set w.r.t. a variety $\classV$ and $\Gamma$ is an r-complete set. Then for every $\idn_i, i=1,\dots,n$, due to r-completeness of $\Gamma$, there is an identity $\idr_i \in \Gamma$ such that $\idn_i \models_\classV \idr_i$. Hence, $\{ \idr_1,\dots,\idr_n \}$ forms an r-complete w.r.t. $\classV$ set. 
\end{proof}

The following theorem establishes relations between r-complete sets of identities and algebras generating the variety.

\begin{theorem} \label{rcompl} Let $\classV$ be a variety, $\classI$ be a set of identities and $\classK \subseteq \classV$ be a class of (at most countable) algebras. Then
\begin{itemize}
\item[(a)] If the set $\classI$ is r-complete w.r.t. $\classV$ and for each $\idn \in \classI, \classK \not\models \idn$ , then $\classK$ generates the variety $\classV$;
\item[(b)] If $\classK$ generates a variety $\classV$ that has a TD term then the set of identities $\set{CS(\AlgA)}{\AlgA \in \classK}$ is r-complete w.r.t. $\classV$.
\item[(c)] If $\classK$ generates $\classV$ and consists of finite s.i. algebras, then the set of characteristic identities $\set{\chi_{_\classV(\AlgA)}}{\AlgA \in \classK}$ is r-complete w.r.t. $\classV$.
\end{itemize}
\end{theorem}

\begin{proof} (a) Assume that $\classI$ is r-complete and for each $\idn \in \classI, \classK \not\models \idn$. Since $\classK \subseteq \classV$, clearly, $\eqc(\classK) \subseteq \classV$ and we need only to show that for every identity $\idt$ if $\classV \not\models \idt$ then $\classK \not\models \idt$. Indeed, due to r-completeness, there is $\idn \in \classI$ such that $\idt \models_\classV \idn$. By the assumption of the theorem, $\idn$ is refutable in some algebra $\AlgA \in \classK$, hence, $\idt$ must be refutable too.

(b) Let $\classK$ generates $\classV$ and $\idn$ be an identity refutable in $\classV$. Then, for $\classK$ generates $\classV$, there is an algebra $\AlgA \in \classK$ in which $\idn$ is refutable, i.e. $\AlgA \not\models \idn$. Now we can apply the Corollary \ref{partref} and conclude that $\idn \models_\classV \chi(\AlgA',x_\alga,x_\algb)$, where $\chi(\AlgA',x_\alga,x_\algb) \in CS(\AlgA)$. Hence, the set $\set{CS(\AlgA)}{\AlgA \in \classK}$ is r-complete w.r.t. $\classV$. 

(c) We can repeat the above argument and use the Corollary \ref{genrefute} instead of Corollary \ref{partref}.  
\end{proof}

\begin{cor} If $\classV$ is a finitely approximated variety with a TD term then $\set{\chi_{_\classV}(\AlgA)}{\AlgA \in FSI(\classV)}$ is r-complete w.r.t. $\classV$. 
\end{cor}

Recall that a variety is called \textit{tabular} if it is generated by a single finite algebra.

\begin{theorem} Let $\classV$ be a finitely approximated variety with a TD term. Then $\classV$ is tabular if and only if there is a finite r-complete w.r.t. $\classV$ set of identities.  
\end{theorem}
\begin{proof} Assume that a variety $\classV$ is tabular and has a TD term. Then, $\classV$ is generated by some finite algebra $\AlgA$. Since $\AlgA$ is finite, it is a subdirect product of some finite s.i. algebras $\AlgA_1,\dots,\AlgA_n$. Clearly, the variety $\classV$ is generated by algebras $\AlgA_1,\dots,\AlgA_n$ and we can apply the Theorem \ref{rcompl} (c) and complete the proof of this case.

Conversely, assume that $\classV$ has a finite r-complete set. By the definition of finite approximability, $\classV$ is generated by its finite algebras. We can safely assume that the variety $\classV$ is generated by some set $\classK$ of finite s.i. algebras. Hence, due to the Theorem \ref{rcompl} (c), the set $\set{\chi_{_\classV}(\AlgA)}{\AlgA \in \classK}$ is r-complete. Now we can apply the Proposition \ref{finrcompl} and conclude that there is such a finite subset $\AlgA_1,\dots,\AlgA_n \in \classK$ that $\lbrace \chi(\AlgA_i), i=1,\dots,n \rbrace$ is r-complete. It is not hard to see that algebras $\AlgA_1,\dots,\AlgA_n$ generate $\classV$, thus, $\classV$ is tabular.
\end{proof}

\subsection{r-Complete Sets and Splitting}

The goal of this Section is to prove that the splitting varieties have an r-complete set that consists of a single identity.

\begin{theorem} \label{refsplitting} Let $\classV$ be a variety and $\AlgA \in \classV$ be an algebra. Then $\AlgA$ is a splitting algebra if and only if the variety $\classV_0 = \eqc(\AlgA)$ generated by $\AlgA$ has an r-complete w.r.t. $\classV$ set that consists of a single identity that is invalid in $\AlgA$. 
\end{theorem}

\begin{proof} Suppose $\AlgA$ is a splitting algebra and $\classV_0 = \eqc(\AlgA)$ is a splitting variety. Let a co-splitting variety $\classV_1$ be defined by an identity $\ids$. Let us prove that $\{ \ids \}$ is an r-complete set.

Suppose $\idn$ is an identity such that $\classV_0 \not\models \idn$. We need to show that $\idn \vDash_\classV \ids$. Assume for contradiction $\idn \nvDash_\classV \ids.$
Then for some algebra $\AlgB \in \classV$ we have
\begin{equation}
\AlgB \models \idn \text{ and } \AlgB \not\models \ids. \label{c-assump1}
\end{equation}
Let us consider the variety $\eqc(\AlgB$) generated by $\AlgB$. Clearly, $\AlgB \notin \classV_1$, for the variety $\classV$ is defined by $\ids$. Hence, $\eqc(\AlgB) \nsubseteq \classV_1$ and, by the definition of splitting algebra, $\AlgA \in \eqc(\AlgB)$. The latter means that $\AlgB \models \idn$ yields 
\[
\AlgB \models \ids
\]
and the latter contradicts \eqref{c-assump1}.

Conversely, let $\classV$ be a subvariety and $\ids$ be an identity that forms by itself an r-complete w.r.t. $\classV$ set and $\AlgA \not\models \ids$. We need to prove that $\AlgA$ is a splitting algebra, that is, we need to prove that there is the largest subvariety of $\classV$ not containing $\AlgA$. Our goal is to demonstrate that the subvariety $\classV_1 \subseteq \classV$ defined by $\ids$ is, indeed, the largest subvariety of $\classV$ not containing $\AlgA$. Obviously, $\AlgA \notin \classV_1$. Now, let us recall that $\{\ids\}$ is an r-complete set. Hence, if some identity $\idn$ is invalid in $\AlgA$, this identity is invalid in $\classV_0$ and, by the definition of r-completeness, we have $\idn \vDash_\classV \ids$. The latter means that $\ids$ is valid in every algebra in which $\idn$ is valid and, hence, every algebra in which $\idn$ is valid belongs to $\classV_1$. 
\end{proof}

We can modify Harrop's algorithm and prove the following theorem.

\begin{theorem} \label{dsplitting} Assume $\classV_0$ is a variety and $\classV \subseteq \classV_0$ is a splitting subvariety. Then if equational logic $Eq(\classV)$ is finitely axiomatizable then $Eq(\classV)$ is decidable.
\end{theorem} 
\begin{proof}
Let $Ax$ be a finite set of axioms of $Eq(\classV)$ and $\ids$ be an identity defining the co-splitting variety. Then in order to determine whether a given identity $\idn$ is derivable from $\Ax$ we need in parallel to do the following:
\begin{itemize}
\item[(a)] to try to derive $\idn$ from $Ax$;
\item[(b)] to try to derive $\ids$ from $Ax \cup \{\idn\}$.
\end{itemize}
It follows from the definition of splitting that one of these two processes will halt in finite number of steps. 
\end{proof}

\subsection{r-Complete Sets in the not Finitely Approximated Varieties}

As we saw in the Section \ref{rcomplete}, if a finitely approximated variety with a TD term has a finite r-complete set, then this variety is tabular. In this section we will prove that if a variety is not finitely approximated, the situation is different.

\begin{theorem} \label{finappdec} Every finitely pre-approximated variety has an r-complete set that consist of a single identity. Hence, the equational logic of every finitely axiomatizable finitely pre-approximated variety is decidable.
\end{theorem}
\begin{proof} Let $\classV$ be a finitely pre-approximated variety. Our goal is to demonstrate that a variety $\classV^\circ$ generated by all finite algebras of $\classV$ is the largest proper subvariety of $\classV$ and, hence, is a co-splitting variety in $\classV$, while the variety $\classV$ splits itself. Then we can apply the Theorem \ref{dsplitting} and complete the proof.

By the definition of finitely pre-approximated variety, the variety $\classV$ is not generated by its finite algebras, hence $\classV^\circ \subsetneq \classV$. Let $\AlgA \in \classV \setminus \classV^\circ$ be a subdirectly irreducible algebra. We will show that 
\begin{itemize}
\item[(a)] $\AlgA$ generates the variety $\classV$;
\item[(b)] $\classV^\circ$ is the largest subvariety of $\classV$ not containing $\AlgA$, 
\end{itemize}
that is, $\AlgA$ is a splitting algebra and $\AlgA$ generates $\classV$. 

(a) Since $\AlgA \notin \classV^\circ$ we have $\eqc(\AlgA) \not\subseteq \classV^\circ$. Hence, the variety $\eqc(\AlgA)$ is not generated by finite algebras. Using the definition of finitely pre-approximated variety we can conclude that $\eqc(\AlgA) = \classV$. 

(b) Let $\classV'$ be a subvariety of $\classV$ and $\AlgA \notin \classV'$. Then $\classV'$ is a proper subvariety of $\classV$ and, by the definition of finitely pre-approximated variety, $\classV'$ is finitely approximated. Hence, the variety $\classV'$ is generated by finite algebras and, therefore, $\classV' \subseteq \classV^\circ$. Thus, every subvariety not containing $\AlgA$ is a subvariety of $\classV^\circ$.
\end{proof}

\begin{remark} In \cite[Theorem 7.2]{Skura1992} T.~Skura proved the decidability of some not finitely approximated varieties of Heyting algebras by using a pre-true identity as a r-complete set.
\end{remark}

Recall (see, for instance, \cite{Mardaev_Embedding_1987}) that a variety $\classV$ is called \textit{pre-locally finite} if $\classV$ is not locally finite, but all proper subvarieties of $\classV$ are locally finite. In  \cite{Mardaev_Embedding_1987} S.~Mardaev proved that there is a continuum of pre-locally finite varieties of intermediate and positive logics.

\begin{theorem} Let $\classV$ be a finitely axiomatized pre-locally finite variety. Then equational logic $Eq(\classV)$ is decidable.
\end{theorem}
\begin{proof} Let $\classV$ be a finitely axiomatized pre-locally finite variety. Then $\classV$ is either finitely approximated, or not. If $\classV$ is finitely approximated, then we can use the Harrop's algorithm, and conclude that $Eq(\classV)$ is decidable. If the variety $\classV$ is not finitely approximated, clearly, it is finitely pre-approximated, for every locally finite variety is finitely approximated. Hence, in this case we can apply the Theorem \ref{finappdec} and complete the proof. 
\end{proof}

\subsection{Finial Remarks}

As we have seen, an ability to construct a characteristic identity is rested on the properties of a TD term. If we move from varieties (logics understood as closed sets of formulas) to quasivarieties (single conclusion consequence relations) or to universal classes (multiple conclusion consequence relations), the situation becomes more simple, because there is no need to use a TD term. In a single conclusion case, a quasi-identity constructed of finitely presented algebra corresponds to a quasi-characteristic rule. Quasi-characteristic rules were introduced by the author in \cite{Citkin1977} for intermediate logic and then were extended to modal logics by V.~Rybakov (see  \cite{Rybakov_Quasi_Characteristic_1997,Rybakov_Book}). The algebraic properties of corresponding quasi-identities are studied in \cite{Budkin_Gorbunov_1975}. 

In the multiple conclusion case, one can repeat the arguments (without reference to a TD term). For partial algebras, similarly to how we arrived to canonical formulas, one can construct the canonical rules \cite{Jerabek_Canonical_2009}.  It is worth noting that if $\AlgA$ is a partial algebra we can also use a negative diagram 
\[
\delta^-(\AlgA) \bydef \set{f(x_{\alga_1},\dots,x_{\alga_n}) \approx x_\algb}{ \algb \neq f(\alga_1,\dots,\alga_n); \algb,\alga_1,\dots,\alga_n \in \algA; f \in \Con  }
\]   
and consider a characteristic rule
\[
\chi(\AlgA) \bydef \delta^+(\AlgA) / \delta^-(\AlgA). 
\]
It is easy to see that if $\chi(\AlgA)$ is refuted in $\AlgB$, then $\AlgA$ is embedded in $\AlgB$ and algebra $\AlgA$ does not even have to be subdirectly irreducible. This is what makes a multiply conclusion case even much more simple that a single conclusion one.

\subsection{Acknowledgments} The author wants to thank A.~Muravitsky for countless discussions that helped while I was working on this paper.


\bibliographystyle{acm}

\def\cprime{$'$}

\end{document}